\documentclass{amsart}
\usepackage{amsmath}
\usepackage{mathtools}
\usepackage{amssymb}
\usepackage{enumerate}
\usepackage{slashed}
\usepackage{graphicx}
\usepackage{graphbox}
\usepackage{caption}
\usepackage{subcaption}
\usepackage{newlfont}
\usepackage{amsrefs}
\usepackage{comment}
\usepackage[normalem]{ulem}
\usepackage{mathtools}
\usepackage{leftidx}
\usepackage{epsfig,fancyhdr,color}
\usepackage{enumerate}
\usepackage{tikz-cd}
\newcommand{\iso}{\cong}
\newcommand{\of}[1]{\left( #1 \right)}
\newcommand{\plumb}[1]{\mathcal{P}\of{#1}} 
\newcommand{\txi}{\pmb{\xi}}
\usepackage{xargs}
\newcommand{\TD}{\mathcal{\D}}
\newcommand{\sub}{\subset}
\newcommand{\from}{\colon}
\newcommand{\n}{\vec{n}}
\newcommand{\tv}{\boldsymbol\theta}
\newcommand{\bphi}{\boldsymbol\phi}

\newcommand{\twist}{\widetilde{\times}}
\newcommand{\bbar}[1]{\overline{#1}}
\DeclareMathOperator{\Det}{Det}
\newcommand{\p}{\boldsymbol{\mathfrak{p}}}
\newcommand{\pc}{\mathfrak{p}}
\DeclareMathOperator{\spn}{span}

	\usetikzlibrary{cd} 
	\usetikzlibrary{math, angles, calc} 
	\usetikzlibrary{decorations.markings} 
	\usetikzlibrary{snakes}
	\tikzset{->-/.style={decoration={markings, mark=at position #1 with {\arrow{>}}},postaction={decorate}}}
	\newcommand{\horizon}{\draw[decoration = {zigzag,segment length = 0.15cm, amplitude = .45mm},decorate,  thin]}

\numberwithin{equation}{section}
\definecolor{NoteColor}{rgb}{1,0,0}

\theoremstyle{plain}

\newtheorem{main}{Main Result}
\newtheorem{mainTheorem}[main]{Theorem}
\newtheorem{mainConjecture}[main]{Conjecture}
\newtheorem*{mainConjecture*}{Conjecture}

\newtheorem{theorem}{Theorem}[section]
\newtheorem*{theorem*}{Theorem}
\newtheorem{lemma}[theorem]{Lemma}
\newtheorem{corollary}[theorem]{Corollary}
\newtheorem{prop}[theorem]{Proposition}

\newtheorem{conjecture}[theorem]{Conjecture}
\newtheorem*{conjecture*}{Conjecture}

\theoremstyle{definition}
\newtheorem{definition}[theorem]{Definition}
\theoremstyle{remark}
\newtheorem{remark}[theorem]{Remark}

\renewcommand{\div}{\operatorname{div}}
\newcommand{\Tr}{\operatorname{Tr}}

\newcommand{\R}{\mathbb{R}}

\newcommand{\B}{B}
\newcommand{\D}{D}
\newcommand{\CP}{\mathbb{CP}}
\renewcommand{\t}{\mathbf{t}}

\newcommand{\G}{\mathcal{G}}

\newcommand{\rank}{\operatorname{rank}}

\newcommand{\into}{\hookrightarrow}

\renewcommand{\v}{\mathbf{v}}
\renewcommand{\u}{\mathbf{u}}
\renewcommand{\i}{\mathbf{i}}
\renewcommand{\j}{\mathbf{j}}
\newcommand{\e}{\mathbf{e}}

\newcommand{\w}{\mathbf{w}}

\newcommand{\Z}{\mathbb{Z}}
\newcommand{\dv}{\operatorname{div}}

\newcommand{\diag}{\operatorname{diag}}

\renewcommand{\q}{\mathbf{q}}
\renewcommand{\r}{\mathbf{r}}

\newcommand{\group}[2]{\left\langle #1 \big | #2 \right \rangle}

\usepackage{subfiles} 

\begin{document}
\title[The Geometry and Topology of
	Stationary Multi-Axisymmetric Black Holes]{The Geometry and Topology of
	Stationary Multi-Axisymmetric Vacuum Black Holes in Higher Dimensions}

\author[Kakkat]{Vishnu Kakkat}
\address{Department of Mathematics \& Department of Physics\\
	Ariel University, Ariel, Israel 40700}
\email{vishnuka@ariel.ac.il}

\author[Khuri]{Marcus Khuri}
\address{
	Department of Mathematics\\
	Stony Brook University \\
	Stony Brook, NY, 11794-3660\\
	USA}
\email{khuri@math.sunysb.edu}
\thanks{M. Khuri acknowledges the support of NSF Grants DMS-1708798, DMS-2104229, and Simons Foundation Fellowship 681443. J. Rainone was partially supported by the Research Training Group under NSF grant DMS-1547145.}

\author[Rainone]{Jordan Rainone}
\address{Department of Mathematics\\
	Stony Brook University \\
	Stony Brook, NY, 11794-3660\\
	USA}
\email{jordan.rainone@stonybrook.edu}

\author[Weinstein]{Gilbert Weinstein}
\address{Department of Mathematics \& Department of Physics\\
	Ariel University, Ariel, Israel 40700}
\email{gilbertw@ariel.ac.il}

\begin{abstract}
Extending recent work in 5 dimensions, we prove the existence and uniqueness of solutions to the reduced Einstein equations for vacuum black holes in $(n+3)$-dimensional spacetimes admitting the isometry group $\mathbb{R}\times U(1)^{n}$, with Kaluza-Klein asymptotics for $n\geq3$. This is equivalent to establishing existence and uniqueness for singular harmonic maps $\varphi\colon \R^3\setminus\Gamma\to SL(n+1,\R)/SO(n+1)$ with prescribed blow-up along $\Gamma$, a subset of the $z$-axis in $\R^3$. We also analyze the topology of the domain of outer communication for these spacetimes, by developing an appropriate generalization of the plumbing construction used in the lower dimensional case. Furthermore, we provide a counterexample to a conjecture of Hollands-Ishibashi concerning the topological classification of the domain of outer communication.
A refined version of the conjecture is then presented and established in spacetime dimensions less than 8.
\end{abstract}
\maketitle


\section{Introduction}

In several recent papers, harmonic maps into symmetric spaces were used to construct solutions of the 5-dimensional Einstein equations with symmetry group $\R\times U(1)^2$. More precisely, in this situation the Einstein vacuum equations reduce to an axially symmetric harmonic map with prescribed singularities from $\R^3$ into the symmetric space $SL(3,\R)/SO(3)$. 
In~\cite{khuriweinsteinyamada2017}, solutions of this problem corresponding to spacetimes which are asymptotically flat were constructed, while in~\cite{khuriweinsteinyamada2018} a similar approach was applied to obtain solutions with Kaluza-Klein and locally Euclidean asymptotics. Furthermore, the absence of conical singularities on the two unbounded axes was also established in \cite{khuriweinsteinyamada2018}. It is important to emphasize, however, that many of these solutions are expected to have conical singularities on at least one of the bounded components of the axis. In \cite{AlaeeKhuriKunduri2019}, existence and uniqueness results were produced for the stationary bi-axisymmetric minimal supergravity equations.
While in~\cite{khurimatsumotoweinsteinyamada}, plumbing of disk bundles was used to analyze the topology of the domain of outer communication (DOC) of these solutions. It is the purpose of the present work to extend these results to $(n+3)$-dimensional vacuum gravity with symmetry group $\R\times U(1)^n$. Similarly to the above, the Einstein vacuum equations in this setting reduce to an axially symmetric harmonic map with prescribed singularities from $\R^3$ to the symmetric space target $SL(n+1,\R)/SO(n+1)$.

A significant motivation for this higher dimensional study is to expand the availability of candidate regular solutions, as well as to expand the range of topologies exhibited. It is expected that in 4 dimensions, all asymptotically flat stationary and axially symmetric vacuum solutions with more than one horizon, the cross-sections of which must be 2-spheres, will have a conical singularity on some bounded component of the axis of rotation. Some results in this direction have been obtained~\cites{weinsteinTAMS,litian,hennigneugebauer}, but a complete resolution is still out of reach. On the other hand, in dimension 5, there are  several known regular solutions other than the $S^3$-horizon Myers-Perry~\cite{MyersPerry} black holes, namely the Emparan-Reall and Pomeransky-Sen'kov black rings~\cite{emparanreall,pomeranskysenkov} having horizon topology $S^1 \times S^2$, the black Saturns~\cite{elvangfigueras} of Elvang-Figueras, as well as the the black bi-rings~\cite{elvangrodriguez} and di-rings~\cites{iguchimishima,evslinkrishnan} found by Elvang-Rodriguez, Evslin-Krishnan, and Iguchi-Mishima. Recent work by Lucietti-Tomlinson concerning the existence of conical singularities may be found in \cite{LuciettiTomlinson1,LuciettiTomlinson2}, see also \cite{khuriweinsteinyamadaJHEP,KhuriWeinsteinYamadaNeg}.
It is reasonable to expect that many more regular solutions may be found in higher dimensions, other than trivial examples obtained for instance by taking products of known solutions with flat tori. The spacetimes that we produce provide a plethora of candidates having an increasing variety of topologies for the domain of outer communication. Moreover, even those solutions with a conical singularity should be of interest, since we expect that one could perturb time slices to obtain initial data, satisfying relevant energy conditions, with outermost apparent horizon and DOC having exotic topologies.

Motivation is also derived from questions regarding the topological classification of the domain of outer communication. Specifically, we address Conjecture~1 in~\cite{hollands2012}, which postulates that under reasonable hypotheses, the topology of a Cauchy slice in the DOC can be obtained by removing the black hole region from the connected sum of a product of spheres with the asymptotic region. We provide a counterexample to this statement, and discuss why the spirit of the conjecture may nevertheless remain valid. We then offer a refined version of the conjecture, and present a proof for spacetime dimensions less than 8.

The methods used here parallel those employed in~\cites{khuriweinsteinyamada2017,khuriweinsteinyamada2018,khurimatsumotoweinsteinyamada} with a number of notable differences which we now point out. The rod structure, an $n$-tuple of relatively prime integers associated with each axis rod, and which determines the combination of the Killing fields that degenerate on that rod, is much more complex than in the 5-dimensional setting where it was merely a pair of relatively prime integers. In particular, the admissibility condition at the corners (points where two axis rods meet), which ensures that the reconstructed spacetime has the structure of a manifold, now involves second determinant divisors. We are thus led to use Smith and Hermite normal forms. Also, the energy estimates for harmonic maps into higher rank symmetric spaces, needed to prove existence, require us to extend the construction of horocyclic coordinates to these more complicated spaces. Finally, the plumbing construction used to analyze the topology of the DOC in 5 dimensions must be generalized in higher dimensions, and involves in addition to the disk bundle integer invariants, a so called `plumbing vector' which describes how neighboring bundles are glued together.

The paper is organized as follows. The next section presents necessary background and states the main results. In Section~\ref{topology}, we apply Smith and Hermite normal forms to describe the rod structures of $T^n$-manifolds. The model map, an approximate solution of the harmonic map problem, is constructed in Section~\ref{modelmap}. While in Section~\ref{horo}, we produce horocyclic coordinates on the symmetric space target and use them to derive energy estimates. The domain of outer communication is analyzed in Section~\ref{plumbing}, using an adaptation of the technique of plumbing from the topology of disk bundles. We conclude with a study of the Hollands-Ishibashi conjecture in Section \ref{sec:class}. 

\section{Background and Main Results}

A connected asymptotically locally Kaluza-Klein stationary vacuum spacetime, with 3, 4, or 5 `large' asymptotically (locally) flat dimensions, will be referred to as \textit{well-behaved} if the orbits of the stationary Killing field are complete, the domain of outer communication (DOC) is globally hyperbolic, and the DOC contains an acausal spacelike connected hypersurface which is asymptotic to the canonical slice in the asymptotic end and whose boundary is a compact cross section of the horizon. These assumptions are used for the reduction of the stationary vacuum equations, and are consistent with \cite{hollands2012}. By \emph{asymptotically locally Kaluza-Klein} we refer to a spacetime which asymptotes to the ideal geometry $\left(\mathbb{R}^{4-s,1}/G \right)\times T^{n+s-2}$, where $T^{n+s-2}$ is a flat torus, $G\subset O(4-s)$ is a discrete subgroup of spatial rotations, and $s\in\{0,1,2\}$. If $G$ is trivial, then the moniker `locally' is removed from the terminology.

Let $(\mathcal{M}^{n+3},g)$, $n\geq 1$ be a well-behaved asymptotically Kaluza-Klein stationary $n$-axisymmetric vacuum spacetime, that is, it admits  $\mathbb{R}\times U(1)^{n}$ as a subgroup of its isometry group. As a consequence of topological censorship \cite{CGS} the orbit space is simply connected, and hence the spacetime metric $g$ may be written in Weyl-Papapetrou coordinates \cite[Theorem 8]{hollands2012} as
\begin{equation}\label{weyl}
g= f^{-1} e^{2 \sigma} (d\rho^2+dz^2)- f^{-1} \rho^2 dt^2+ f_{ij}(d\phi^i+v^idt)(d\phi^j+v^jdt),
\end{equation}
where $\left(f_{ij}\right)$ is an $n\times n$ symmetric positive definite matrix with determinant $f$, and $f_{ij}$, $v^j$, $\sigma$ are all functions of $\rho$ and $z$. Let
\begin{equation}
g_3= e^{2 \sigma} (d\rho^2+dz^2)- \rho^2 dt^2, \quad\quad \quad A^{(j)}=v^j dt,
\end{equation}
then the vacuum equations imply
\begin{equation}
d(ff_{ij} \star_3dA^{(j)})=0,
\end{equation}
where $\star_3$ represents the Hodge dual operator with respect to $g_3$. Thus, there exist globally defined \emph{twist potentials} $\omega_{i}$ such that
\begin{equation}\label{twist}
d\omega_i=2ff_{ij} \star_3dA^{(j)}.
\end{equation}
The value of the twist potentials on axes adjacent to the horizons determines the angular momenta of the black holes. Next, note that we can write the 3-dimensional reduced Einstein-Hilbert action \cite{maison1979ehlers} as
\begin{equation}\label{Action}
S=\int_{\mathbb{R}\times\left(\mathcal{M}^{n+3}/[\mathbb{R}\times U(1)^n]\right)} R^{(3)}\star_31+\frac{1}{4}\Tr(\Phi^{-1}d\Phi \wedge \star_3\Phi^{-1}d\Phi),
\end{equation}
where
\begin{equation}
\Phi=
\begin{pmatrix}
f^{-1}&  -f^{-1}\omega_i\\
-f^{-1}\omega_i  & f_{ij}+f^{-1}\omega_i\omega_j
\end{pmatrix},
\quad
i,j=1,..,n,
\end{equation}
is symmetric, positive definite, and satisfies $\det(\Phi)=1$. By varying the action with respect to $\Phi$ and applying $\R$-symmetry, a majority of the reduced Einstein vacuum equations may be obtained:
\begin{align}\label{harmonic1}
\begin{split}
\tau^{f_{lj}}=\Delta f_{lj}-f^{km}\nabla^\mu f_{lm}\nabla_\mu f_{kj}+f^{-1}\nabla^\mu \omega_l\nabla_\mu \omega_j=0,\\
\tau^{\omega_j}=\Delta \omega_{j}-f^{kl}\nabla^\mu f_{jl}\nabla_\mu \omega_{k}-f^{lm}\nabla^\mu f_{lm}\nabla_\mu \omega_j=0.
\end{split}
\end{align}
These are the equations for a harmonic map $\varphi:\R ^3\setminus \Gamma\to SL(n+1,\R)/SO(n+1)$. Given a solution to this system, the remaining metric components $v^i$ and $\sigma$ may be found~\cite{shiromizu} by quadrature. Therefore, the stationary vacuum equations in the $n$-axially symmetric setting are equivalent to a harmonic map problem with prescribed singularities on $\Gamma$, a subset of the $z$-axis which represents the axes of the $U(1)^n$-action or rather those points associated with a nontrivial isotropy group.

Consider the orbit space $\mathcal{M}^{n+3}/[\mathbb{R}\times U(1)^n]$. It is homeomorphic to the right half plane $\{(\rho,z):\rho>0\}$, and its boundary $\rho=0$ encodes the topology of the horizons \cite{harmark2004stationary,hollandsyazadjiev,hollands2011}. The domain for the harmonic map is obtained from this observation by adding an ignorable angular coordinate $\phi\in [0,2\pi)$, yielding $\mathbb{R}^3$ parametrized by the cylindrical coordinates $(\rho,z,\phi)$. The harmonic map itself is axisymmetric, as it does not depend on $\phi$. Uniqueness theorems for higher dimensional stationary $n$-axisymmetric black holes ultimately reduce to the uniqueness question for such harmonic maps \cite{hollands2011}, with prescribed axis behavior determined by invariants called \emph{rod structures} as well as a set of
\emph{potential constants}; see Section~\ref{rodstructure} below for details. Together this information forms a \emph{rod data set}, which may be encoded in an approximate solution referred to as a \emph{model map}. We then say that the model map \emph{corresponds} to the rod data set.
If the rods that represent horizon cross-sections have nonzero length, then the rod structure is associated with nondegenerate black hole solutions
\cite[Lemma 7]{hollands2011}.
The prescribed harmonic map problem is solved by finding a solution which is \emph{asymptotic} to the model map. A precise description of the properties required for the model map is given in Definition \ref{modeldef}, and the notion of asymptotic maps is reviewed in Definition \ref{asymptotic}. Our first main result is a generalization of Theorem~1 in~\cite{khuriweinsteinyamada2017}. In particular, it extends the previous result to higher dimensions, and removes the assumption of a compatibility condition for the rod data. However the notion of \emph{admissibility}, which is explained in Section \ref{rodstructure}, is still retained since this is required to ensure that the total space arising from the rod structures is a manifold.


\begin{mainTheorem}\label{theorem1}\

\begin{enumerate}[(a)]
\item For any admissible rod data set, with nondegenerate horizon rods, there exists a model map $\varphi_0:\mathbb{R}^3\setminus \Gamma\to SL(n+1,\mathbb{R})/SO(n+1)$ which corresponds to the rod data set.

\item There exists a unique harmonic map $\varphi:\mathbb{R}^3\setminus \Gamma\to SL(n+1,\mathbb{R})/SO(n+1)$ which is asymptotic to the model map $\varphi_0$.

\item A well-behaved asymptotically (locally) Kaluza-Klein solution of the $(n+3)$-dimensional vacuum Einstein equations admitting the isometry group $\mathbb{R}\times U(1)^n$ can be constructed from $\varphi$ if and only if the resulting metric coefficients are sufficiently smooth across $\Gamma$, and there are no conical singularities on any bounded axis rod.
\end{enumerate}
\end{mainTheorem}

Consider now the topology of the domain of outer communication. In 5 dimensions, we obtained a classification theorem \cite[Theorem 1]{khurimatsumotoweinsteinyamada} in which the canonical
slice was decomposed into a disjoint union of linearly plumbed disc bundles over 2-spheres, and a few other more simple pieces. There does not seem to be a direct natural generalization of linear plumbing which is applicable to the higher dimensional setting of stationary $n$-axisymmetric vacuum spacetimes. In fact, a naive approach leads to a construction that is not unique, as there are various ways to glue the neighboring toroidal fibers together.
In order to remedy this issue we define a generalized or \emph{toric plumbing} with additional parameters $\p_i \in\mathbb{Z}^n$ which are called \emph{plumbing vectors}, see Definition~\ref{def:plumbing}.
In the next result, the higher dimensional generalization of \cite[Theorem 1]{khurimatsumotoweinsteinyamada} is presented. This theorem applies beyond the realm of vacuum solutions, namely to those satisfying the null energy condition,
which is a hypothesis included to ensure that the topological censorship theorem \cite[Theorem 5.3]{CGS}, \cite[Theorem 5]{hollands2012} is valid.

We will use the following notation for the building blocks of the decomposition. The axis $\Gamma$ is a union of intervals $\{\Gamma_{i,j}\}_{i=1}^{I_j +2}$, $j=1,\ldots,\mathfrak{J}$ called \emph{axis rods}, each of which is defined by a particular isotropy subgroup of $U(1)^n$. With each such rod that is flanked on both sides by another axis, we associate $\txi_{i,j} =\xi_{i,j} \times T^{n-3}$ where $\xi_{i,j}$ is a ($\D^2$) disc-bundle over either the 3-sphere $S^3$, the ring $S^1 \times S^2$, or a lens space $L(p,q)$ with $p>q$ relatively prime positive integers.  A sequence of such product spaces may be glued together, with the help of plumbing vectors, to form the toric plumbing $\plumb{\txi_{1,j} , \ldots , \txi_{I_j ,j} \Big | \p_{2,j},  \ldots, \p_{I_j,j}}$. The topologies of $\xi_{i,j}$, and the plumbing vectors themselves $\p_{i,j}$, are completely determined by the rod structures of the axes involved.

\begin{mainTheorem}
\label{main:DOC}
The topology of the domain of outer communication of an orientable well-behaved asymptotically Kaluza-Klein stationary $n$-axisymmetric spacetime, with $n\geq 3$, and satisfying the null energy condition is $\mathcal{M}^{n+3}=\mathbb{R}\times M^{n+2}$ where the Cauchy surface is given by a union of the form
\begin{equation}\label{gfuiqihgkjso}
M^{n+2} = \bigcup_{j=1}^J
 \plumb{\txi_{1,j} , \ldots , \txi_{I_j ,j} \Big | \p_{2,j},  \ldots, \p_{I_j,j}}
  \bigcup_{k=1}^{N_1} C^{n+2}_k
  \bigcup_{m=1}^{N_2}\B^4_m\times T^{n-2}
  \bigcup M_{\text{end}}^{n+2},
\end{equation}
in which each constituent is a closed manifold with boundary and all are mutually disjoint expect possibly at the boundaries. Here $C_k^{n+2}$ is $[0,1]\times \D^2\times T^{n-1}$, $B^4_{m}$ denotes a 4-dimensional ball, and the asymptotic end $M_\text{end}^{n+2}$ is given by $\R_+\times Y\times T^{n-2}$ where $Y$ represents either $S^3$, or $S^1 \times S^2$. Furthermore $J$, $N_1$, and $N_2$ are the number of connected components of the axis which consist of three or more axis rods, one finite axis rod, and two axis rods, respectively.
\end{mainTheorem}

This result identifies the fundamental constituents of the DOC, and its proof shows how they may be computed from the rod structure of the torus action. On the other hand, it does not express the topology in a concise way. In order to achieve this goal, at least in low dimensions, we observe in the next result that
a simplified expression may be obtained by filling in the horizons and capping off the asymptotic end with appropriately chosen toric plumbings. In particular, this produces a `compactified DOC' which is a simply connected $(n+2)$-manifold without boundary admitting an effective $T^n$-action. Classification results for such manifolds \cite{orlik1970actions, Oh5dim, Oh6dim} may then be applied to obtain the following theorem, which generalizes ~\cite[Theorem 2]{khurimatsumotoweinsteinyamada} where the case $n=2$ was treated.

\begin{mainTheorem}
\label{main:456classification}
Consider the domain of outer communication $\mathcal{M}^{n+3}=\mathbb{R}\times M^{n+2}$ of an orientable well-behaved asymptotically Kaluza-Klein stationary $n$-axisymmetric spacetime, with $2\leq n\leq 4$, satisfying the null energy condition, and having $H$ components of the horizon cross-section. There exists a choice of horizon fill-ins $\{\bar{M}_{h}^{n+2}\}_{h=1}^{H}$ and a cap for the asymptotic end $\bar{M}_{\text{end}}^{n+2}$, each of which is either the product of a 4-ball with a torus $\B^4\times T^{n-2}$ or a finite toric plumbing, such that the compactified Cauchy surface
\begin{equation}
\bar{M}^{n+2}=\left(M^{n+2}\setminus M^{n+2}_{\text{end}}\right)\bigcup_{h=1}^{H}\bar{M}_{h}^{n+2}\bigcup
\bar{M}^{n+2}_{\text{end}}
\end{equation}
is homeomorphic to one of the following possibilities, where $k=b_2(\bar{M}^{n+2})$ is the second Betti number and $0\leq \ell\leq k$.
\begin{center}\label{ahkgka}
\begin{tabular} {|c|c|c|}
$n=2$ & $n=3$ & $n=4$ \\
\hline
$S^4$
& $S^5$
& $S^3\times S^3$ \\

$\# \tfrac{k}{2} (S^2 \times S^2)$
& $ \# k (S^2\times S^3)$
& $\# k(S^2\times S^4)\# (k+1)(S^3\times S^3)$ \\

$\ell\mathbb{CP}^2 \# (k-\ell)\bbar{\CP^2}$
& $(S^2\twist S^3)  \# (k-1) (S^2\times S^3)$
& $(S^2\twist S^4)  \# (k-1)(S^2\times S^4)\# (k+1)(S^3\times S^3)$ \\
\end{tabular}
\end{center}
\medskip
\noindent Moreover, the toric plumbings for each fill-in and cap may be computed algorithmically from the neighboring rod structures of each horizon and the asymptotic end.
\end{mainTheorem}

In the chart above, the first row consists of the case when the compactified DOC is $2$-connected, while the second and third rows consist of the spin and non-spin scenarios, respectively. In the second and third rows the second Betti number $k$ is positive, and is even for dimension 4 with the spin property. The twisted product notation is used to denote the nontrivial (and non-spin) sphere bundles over $S^2$. Furthermore, note that $S^2\twist S^2 \cong\CP^2\#\bbar{\CP^2}$
and $\CP^2\#\bbar{\CP^2}\#\CP^2 \cong \CP^2\# (S^2\times S^2)$~\cite[Remark 5.8]{orlik1970actions}. This together with \cite[Theorem II.4.2, pg. 313]{Pau}, shows that in the non-spin 4 dimensional case an alternate expression for the decomposition may be given in terms of a connected sum of
a number of $S^2 \times S^2$'s, and either a single $S^2\twist S^2$ or a number of $\mathbb{CP}^2$'s. This is analogous to the result for dimensions 5 and 6 modulo the presence of the complex projective planes. Theorem \ref{main:456classification} may be thought of as evidence towards a modified version of a conjecture made by Hollands and Ishibashi in \cite[Conjecture 1]{hollands2012}, concerning the topological classification of the DOC under a spin assumption. In Section \ref{sec:class} we construct a spacetime which serves as a counterexample to the original conjecture, and
this motivates the refinement below. Note that Theorem \ref{main:456classification} shows that the following conjecture holds true for $n=2,3,4$, if the compactified DOC is spin.

\begin{mainConjecture}
\label{main:refined}
Consider the domain of outer communication $\mathcal{M}^{n+3}=\mathbb{R}\times M^{n+2}$ of an orientable well-behaved asymptotically Kaluza-Klein stationary $n$-axisymmetric spacetime, with $n\geq 2$, satisfying the null energy condition. If the Cauchy surface $M^{n+2}$ is spin, then there exists a choice of horizon fill-in and a cap for the asymptotic end, such that the corresponding compactified DOC
is homeomorphic to
\begin{equation}
\#^{n}_{i=2} m_i \cdot S^i\times S^{n+2-i}
\end{equation}
for some nonnegative integers $m_i$.
\end{mainConjecture}


\renewcommand{\B}{\mathcal{B}}
\renewcommand{\D}{\mathcal{D}} 



\section{Topology and the Rod Structure}
\label{rodstructure}

\renewcommand{\B}{B}
\renewcommand{\B}{D}

\label{topology}


The topology of the spacetimes considered here will always be of the form $\R\times M^{n+2}$, due to the assumption of global hyperbolicity. The time slice $M^{n+2}$ is assumed to admit an effective action by the torus $T^n$, and hence the quotient map $M^{n+2}\to M^{n+2}/T^n$ exhibits $M^{n+2}$ as a $T^n$-bundle over a $2$-dimensional base space with possibly degenerate fibers on the boundary.
Fibers over interior points are $n$-dimensional, while fibers over points along the boundary can be $(n-1)$ or $(n-2)$-dimensional. The set of points where the fiber is $(n-1)$-dimensional are called \emph{axis rods} while the points with an $(n-2)$-dimensional fiber are called \emph{corners}. The set of corners is always discrete. If in addition topological censorship holds, as is the case under the hypotheses of the main theorems, then the base space $M^{n+2}/T^n$ is homeomorphic to a half plane~\cite{hollands2011}. The boundary $\partial\R^2_+$ of this half-plane is divided into disjoint intervals separated by corners or \emph{horizon rods} where the fibers do not degenerate. The boundary points of horizon rods are called \emph{poles}. Associated to each axis rod interval $\Gamma_i\sub \partial \R^2_+$ is a vector $\v_i\in \Z^n$ called the \emph{rod structure}, that defines the 1-dimensional isotropy subgroup $ \R/\Z \cdot \v_i \sub \R^n / \Z^n \iso T^n$ for the action of $T^n$ on points that lie over $\Gamma_i$. The topology of the DOC is determined by the rod structures, namely
\begin{equation}
M^{n+2}\iso (\R^2_+\times T^n)/\sim
\label{eq:T^n}
\end{equation} 
where the equivalence relation $\sim$ is given by $(\mathbf{p},\bphi) \sim (\mathbf{p}, \bphi + \lambda \v_i )$ with $\mathbf{p}\in \Gamma_i$, $\lambda\in \R/\Z$, and $\bphi \in T^n$. This setting is a special case of the following construction.

\begin{definition}
\label{def:T^n-space}
A \emph{simple $T^n$-manifold} is an orientable smooth manifold $M^k$, $k\geq n$ with an effective $T^n$-action, in which the quotient space $M^k/T^n$ is simply connected and the quotient map defines a trivial fiber bundle over the interior of the quotient.
\end{definition}

If $M^{n+2}$ is a simply connected $T^n$-manifold (it admits an effective $T^n$-action) such that $\partial(M^{n+2}/T^n)\ne\emptyset$, then it is necessarily a simple $T^n$-manifold, see Theorem \ref{thrm:simple}.
As above, the topology of an $(n+2)$-dimensional simple $T^n$-manifold is completely determined by the set of rod structures. A graphical representation of this information is called a \emph{rod diagram}, see Figure~\ref{fig:exRodDiagram} for examples. These are drawn as either a disk in the compact case, or a half plane in the noncompact case, in which the boundary is divided into segments with associated rod structure vectors indicating the linear combination of generators that degenerate at the axes. Black dots represent corners or poles where two rods meet, and the segments drawn with jagged lines are horizon rods along which the torus action is free. We will revisit this figure after Lemma \ref{thrm:Hermite}.

It should be noted that the notion of rod structures given above does not guarantee a unique presentation. Indeed, the vectors $\v$ and $2\v$ both generate the same isotropy subgroup $\R/\Z \cdot \v$, and thus both can be used to describe the same rod structure. In order to identify a unique presentation (up to a choice of sign), it is natural to restrict attention to primitive elements. A vector or a set of vectors $\{\v_1, \ldots, \v_k \} \sub \Z^n$ forms a \emph{primitive set}, if they are linearly independent and
\begin{equation}
\Z^n \cap \operatorname{span}_\R\{\v_1, \ldots, \v_k \} = \operatorname{span}_\Z\{\v_1, \ldots, \v_k \}.
\end{equation}
For a single vector $\v=(v_1,\dots,v_n)$, this is equivalent to the components being relatively prime, that is $\gcd\{v_1,\dots,v_n \}=1$. Next, observe that the group $GL(n,\Z)$ of unimodular matrices provides the group of coordinate transformations for $T^n=\R^n/\Z^n$. Two rod diagrams are equivalent if every rod structure of one is obtained from the corresponding rod structure of the other by the action of the same unimodular matrix. Thus, quantities depending only on the $T^n$-structure will be invariant under $GL(n,\Z)$ transformations. The following proposition exhibits an example of such a quantity, $\Det_k$, referred to as the \emph{$k^\text{th}$ determinant divisor} \cite[Chapter II, Section 14]{newman_1972}. In the statement we will use the multi-index notation $I^n_k$, for $k\leq n$, to denote the set of $k$-tuples $\i=(i_1,\dots,i_k)\in\Z^k$ such that $1\leq i_1<\dots<i_k\leq n$.

\begin{figure}\centering
\includegraphics[width=7cm]{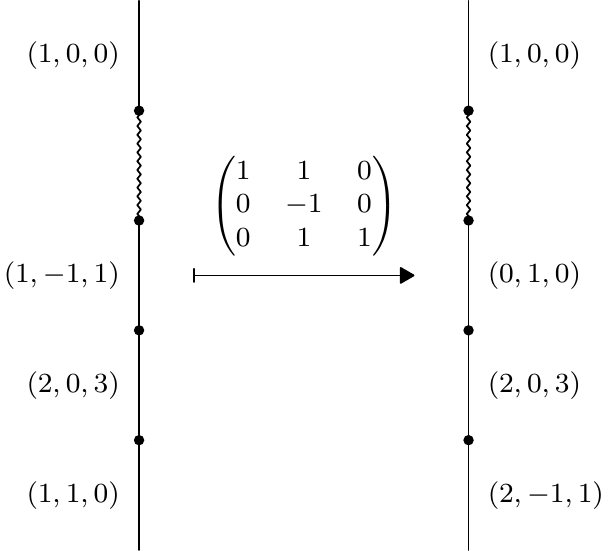}
\caption[caption]{This figure shows two rod diagrams, separated by an arrow, both depicting $(5+1)$-dimensional spacetimes with a single black hole. Each rod diagram shows the $2$-dimensional quotient space as the right-half-plane with the vertical lines being their boundaries. The jagged lines are black hole horizon rods, the interior of which correspond to the product of an open interval with $T^3$. The rod structures flanking the horizon rod yield horizon cross-sectional topology $S^1\times S^3$. The two rod diagrams depict the same spacetime. The unimodular matrix in the middle represents a coordinate change on $T^n$. In particular, it is the transformation matrix from Lemma \ref{thrm:Hermite} which sends the rod structures on the left to their Hermite normal form on the right.}
\label{fig:exRodDiagram}
\end{figure}


\begin{prop}\label{prop1}
Let $\v_1,\dots,\v_m\in\Z^{n}$, $k\leq \mathrm{min}\{m,n\}$, and set
\begin{equation}
\Det_k(\v_1,\dots,\v_m)=\gcd\{Q_{\j}^{\i}\mid \i\in I^n_k, \> \j\in I^m_k\},
\end{equation}
where $Q_{\j}^{\i}$ is the determinant of the $k\times k$ minor obtained from the matrix defined by the column vectors $\v_{1},\dots,\v_{m}$, by picking columns $\j$ and rows $\i$. Then $\Det_k$ is invariant under $GL(n,\Z)$, that is
\begin{equation}
\Det_k(\v_1,\dots,\v_m) = \Det_k(A\v_1,\dots,A\v_m )
\end{equation}
for all $A\in GL(n,\Z)$.
\end{prop}

\begin{proof}
Let $\omega\in\bigwedge^k\mathbb{Z}^n$ be a $k$-form on $\mathbb{Z}^n$. Each such form can be written as a linear combination of the basis elements
$\{\e^{i_1}\wedge\dots\wedge \e^{i_k} \mid \i\in I^n_k\}$, where $\{\e^i\}$ is the basis of covectors dual to the standard basis $\{\e_{j}\}$ of $\Z^n$, so that $\e^{i}(\e_{j})=\delta^i_j$. Thus
\begin{equation} \label{forms}
\omega=\sum_{\i\in I^n_k} a_{i_1\dots i_k}\e^{i_1}\wedge\dots\wedge \e^{i_k}, \quad\quad a_{\i}\in\Z,
\end{equation}
where by definition $\e^{i_1}\wedge\dots\wedge \e^{i_k}(\v_{j_1},\dots,\v_{j_k})$ is the minor determinant $Q^{\i}_{\j}$. Consider the $k\times k$ minor determinant  ${Q'}^{\i}_{\j}$ of the matrix formed from the column vectors $A\v_{j_1},\dots,A\v_{j_k}$, and observe that ${Q'}^{\i}_{\j}$ is multilinear and antisymmetric in $\{\v_{j_1},\dots,\v_{j_k}\}$. Therefore it is a linear combination as in~\eqref{forms}, and may be expressed as
\begin{equation}
{Q'}^{\i}_{\j}=\sum_{\i '\in I^n_k} a^{\i}_{\i '}{Q}^{\i '}_{\j}.
\end{equation}
Observe that if $p\in \Z$ divides ${Q}^{\i '}_{\j}$ for all $\i'\in I^n_k$, then $p$ also divides ${Q'}^{\i}_{\j}$ and hence
\begin{equation}\label{P.1}
\Det_k(A\v_1,\dots,A\v_m )=\gcd\{{Q'}^{\i}_{\j}\mid \i\in I^n_k, \> \j\in I^m_k\}\geq\gcd\{Q^{\j '}_{\i}\mid \i '\in I^n_k, \> \j\in I^m_k \}=\Det_k(\v_1,\dots,\v_m).
\end{equation}
Furthermore since $A^{-1} \in GL(n,\mathbb{Z})$, the same reasoning shows that
\begin{equation}\label{P.2}
\Det_k(\v_1,\dots,\v_m )=	\Det_k(A^{-1}(A\v_1),\dots,A^{-1}(A\v_m) )\geq \Det_k(A\v_1,\dots,A\v_m ).
\end{equation}
The desired invariance follows from these two inequalities.
\end{proof}



A corner point between two adjacent axis rods is \emph{admissible} if the total space over a neighborhood of the corner is a manifold. The importance of the second determinant divisor in the current context arises from the fact that it determines whether or not a corner is admissible.  Since the corner point represents an $(n-2)$-torus within the total space, a tubular neighborhood will be a manifold if and only if it is homeomorphic to $B^4\times T^{n-2}$, or equivalently if its boundary is $S^3 \times T^{n-2}$. This last criteria occurs precisely when there is a matrix $Q\in GL(n,\Z)$ such that $Q\v=\e_1$ and $Q\w=\e_2$, where $\v$, $\w$ are the rod structures of the axis rods forming the corner, and $\e_1$, $\e_2$ are members of the standard basis for $\Z^n$. Corollary \ref{thrm:admissible} below, guarantees that such a $Q$ exists if and only if $\Det_2 (\v, \w)=1$. The statement of this result uses the \emph{Hermite normal form}, whose properties are listed in the next lemma. A proof of this lemma can be found in~\cite{Mader}. The Hermite normal form may be viewed as the integer version of the reduced echelon form, or as the integer version of the $QR$ decomposition for real matrices.


\begin{lemma}\label{Hermitenormalform}
Let $A$ be a $n\times k$ integer matrix. There exist integer matrices $Q$ and $H$ such that $QA=H$, where $Q$ is unimodular and $H=\left( h_{ij}\right)$ has the following properties.
\begin{enumerate}
\item For some integer $m$, the rows $1$ through $m$ of $H$ are non-zero, and the rows $m+1$ through $n$ are rows of zeros.

\item There is a sequence of integers $1\leq r_1 < r_2 <\cdots < r_m \leq r=\rank A$ such that the entries $h_{ir_i}$ of $H$, called \emph{pivots}, are positive for $i=1, \ldots, m$. The pivot $h_{i r_i}$ is the first non-zero element in the row $i$, that is, $h_{ij}=0$ for $1\leq j<r_i$.

\item In each column of $H$ that contains a pivot, the entries of the column are bounded between $0$ and the pivot, that is, for $i=1, \ldots, m$ and $1\leq j<i$ we have $0\leq h_{jr_i}<h_{ir_i}$.
\end{enumerate}
The matrix $H$ is unique and is known as the \emph{Hermite normal form} of $A$. Furthermore, the Hermite normal form of $BA$ is equal to the Hermite normal form of $A$ whenever $B$ is a unimodular matrix. Finally, the unimodular matrix $Q$, known as the \emph{transformation matrix} of $A$, is unique when $A$ is an invertible square matrix.
\label{thrm:Hermite}
\end{lemma}

It should be noted that if the first $l$ columns of $A$ are linearly independent, then the upper-left $l\times l$ block of the Hermite normal form of $A$ is upper triangular with nonzero diagonal entries, namely $r_i=i$ for $i=1, \ldots, l$. For our purposes, the matrix $A$ will typically consist of a collection of $k$ rod structures for rods which are not necessarily adjacent. An example of this is shown in Figure~\ref{fig:exRodDiagram}, where the $3\times 4$ matrix $A$ is assembled from the rod structures on the left (treated as column vectors), and sent to its Hermite normal form consisting of the rod structures on the right, via the transformation matrix that appears in the middle of the diagram.

\begin{remark}
\label{rmk:simple}
If rod structures $\{\v_1, \v_2, \v_3\}$ arise from three consecutive rods with admissible corners, then more information is known about their Hermite normal form $\{\w_1, \w_2, \w_3\}$. In particular $\w_1 = \e_1$, $\w_2 = \e_2$, and $\w_3 = (q, r, p, 0, \ldots, 0)$ with $0\leq q< p$, $0\leq r < p$, $p = \Det_3(\v_1, \v_2, \v_3)$, and $\gcd\{q,p\}=1$ if the set of vectors is linearly independent. In the case of a linearly dependent triple, we have $p=0$ and $q=1$, while $r$ is unconstrained. Furthermore, given any integers $\mu,\lambda\in \Z$ there exists a coordinate change which sends $\v_i$ to $\w'_i$ where
\begin{equation}
\label{eq:qrp}
\begin{split}
\w'_1 & =(1,0, \ldots, 0) \\
\w'_2 & = (0,1,0, \ldots, 0)\\
\w'_3 & = (q+ \mu p, r + \lambda p, p, 0, \ldots, 0)\text{.}
\end{split}
\end{equation}
These observations will be utilized in Section~\ref{plumbing}.
\end{remark}

In order to establish the relationship between the admissibility condition for corners and the 2nd determinant divisor, we recall the \emph{Smith normal form}. This may be considered as the integer matrix analog of the singular value decomposition, and is utilized in the classification of finitely generated Abelian groups. This latter fact will be employed when we compute the fundamental group of the DOC in Theorem~\ref{thrm:simple}. A proof of the following result can be found in~\cite{newman_1972}.

\begin{lemma}
Let $A$ be an $n\times k$ integer matrix of rank $l$. There exist integer matrices $U$, $V$, 
and $S$ such that $UAV=S$. The matrices $U$ and $V$ are unimodular, and $S$ is diagonal with entries $s_i$
such that $s_i|s_{i+1}$ for $1\leq i<l$. These entries, referred to as elementary divisors, satisfy $s_i=0$ 
for $i>l$ with all others computed by
\begin{equation}\label{eq:si}
s_i = \frac{\Det_i(A)}{\Det_{i-1}(A)}, \quad\quad i\leq l,
\end{equation}
where we have set $\Det_0(A)=1$. The matrix $S$ is unique and is known as the \emph{Smith normal form} of $A$.
\label{thrm:Smith}
\end{lemma}

The distinction between the Hermite and Smith normal forms, in the context of rod structures, is as follows.
The transformations used to obtain Hermite normal form are always actions by $n\times n$ matrices on the left. Such an action corresponds to shuffling the Killing vectors around by linear combinations. This does not affect the topology of the total space nor its toric structure, only the representation of the torus $T^n\cong \R^n/\Z^n$ and thus the rod structures.
By contrast, Smith normal form also includes actions on the right by $k\times k$ matrices. These actions correspond to shuffling the axis rods themselves. This changes the topology of our space, possibly no longer making it a manifold.
Consequently, when seeking out a simpler presentation of the rod structures we will invoke the Hermite normal form in order to avoid changing the topology. Two exceptions to this are in the proof of Theorem~\ref{thrm:simple}, where only the integer span of the rod structures is significant and not their order, and in the proof of Corollary \ref{thrm:admissible} below, where the Hermite and Smith normal forms coincide.



\begin{corollary}
Let $A$ be an $n\times k$ integer matrix of rank $k$. Then $\Det_k(A)=1$ if and only if the upper $k\times k$ block of the the Hermite normal form of $A$ is the identity matrix.
\label{thrm:admissible}
\end{corollary}

\begin{proof}
Assume that the upper $k\times k$ block of the Hermite normal form is the identity. By uniqueness, this matrix is also the Smith normal form. The diagonal entries are then $1=s_i= \Det_i(A)/\Det_{i-1}(A)$, which implies that $\Det_k(A)=\Det_{k-1}(A)=\cdots =\Det_0(A)=1$.

Conversely, assume that $\Det_k(A)=1$ and let
\begin{equation}
\begin{bmatrix} S \\  0 \end{bmatrix}=UAV	
\end{equation}
be the Smith normal form of $A$, where $S=\diag(s_1,\dots,s_k)$. Consider the $n\times n$ matrix
\begin{equation} 
B = U^{-1}
\begin{bmatrix}
S & 0 \\
0 & \mathbf{I}_{n-k}
\end{bmatrix}
\begin{bmatrix}
V^{-1} & 0 \\
0 &  \mathbf{I}_{n-k}
\end{bmatrix} = \begin{bmatrix} A & E \end{bmatrix},
\end{equation}
where $E$ consists of the last $n-k$ columns of $U^{-1}$. It follows that
\begin{equation}
\det(B)=\det(U^{-1})\det(S)\det(V^{-1})=s_1\cdots s_k = \frac{\Det_1(A)}{\Det_0(A)}\cdots \frac{\Det_k(A)}{\Det_{k-1}(A)}=\Det_k(A).
\end{equation}
By assumption $\Det_k(A)=1$, and thus $B$ is invertible. Therefore
\begin{equation}
B^{-1}A=\begin{bmatrix} \mathbf{I}_k \\ 0 \end{bmatrix},
\end{equation}
and by uniqueness this must be the Hermite normal form of $A$.
\end{proof}

As mentioned after the proof of Proposition~\ref{prop1}, this corollary shows that a pair of adjacent rod structures $\v$, $\w$ is admissible if and only if $\Det_2(\v,\w)=1$. Moreover, in a similar manner, a collection of $k$ rod structures $\{\v_1, \ldots, \v_k\}$ can be sent to the standard basis $\{\e_1, \ldots , \e_k\}$, and thus forms a primitive set, if and only if $\Det_k(\v_1,\dots,\v_k)=1$. Another application of the Hermite normal form is to give a variant proof of Hollands and Yazadjiev's horizon topology theorem \cite[Theorem 2]{hollands2011}. It states that for $n\geq 2$, all closed $(n+1)$-manifolds with an effective $T^n$-action, whose quotient is not a circle, must be a product of $T^{n-2}$ and either $S^3$, a lens space $L(p,q)$, or $S^1\times S^2$. This is a generalization of a result by Orlik and Raymond for $3$-manifolds, see \cite[Section 2]{orlik1970actions}. Observe that the $(n+1)$-dimensional case can be reduced to the $3$-dimensional case by applying the transformation matrix from Lemma \ref{thrm:Hermite} to the matrix of rod structures defining the horizon, which we assume to be primitive vectors. In particular, the resulting Hermite normal form consists of the new rod structures $(1,0, \ldots, 0)$ and $(q, p, 0, \ldots, 0)$, with $0\leq q<p$. With this representation of the $T^n$-action, the last $n-2$ coordinate Killing fields clearly never vanish. Therefore the total space is homeomorphic to a product of $T^{n-2}$, and a $3$-manifold $\Sigma$ with an effective $T^2$ action. According to the possibilities given for the 3-dimensional case, we find that $\Sigma$ is either $S^3$ if $p=1$, $S^1\times S^2$ if $p=0$, or the lens space $L(p,q)$ if $p>1$.

\begin{remark}
\label{rmk:horizontop}
Given a horizon topology $\Sigma\times T^{n-2}$, it is possible to determine the topology of $\Sigma$ directly from the 2nd determinant divisor. Let $\v,\w\in \Z^n$ be primitive vectors that describe the flanking rod structures of the horizon, and compute $\Det_2(\v, \w)$. If this value is $0$, then $\v=\w$ and $\Sigma=S^1\times S^2$. If it is $1$, then the pair is admissible and $\Sigma=S^3$. If $\Det_2(\v, \w)=p>1$ then $\Sigma=L(p,q)$ for some $q<p$. Moreover, $q$ may be found from the relation $\w = q\v \mod p$.
\end{remark}


\begin{theorem}
Given any two (primitive) rod structures $\v$ and $\w$, it is always possible to find a finite number of additional rod structures that connect $\v$ to $\w$ in such a way that each corner in the resulting sequence of rods is admissible. That is, there exists a sequence of rod structures $\{\v_1,\ldots,\v_k\}$, with $\v_1=\v$ and $\v_k=\w$, having the property that $\Det_2(\v_i,\v_{i+1})=1$ for $i=1,\dots, k-1$.
\label{thrm:fillin}
\end{theorem}

\begin{proof}
By Lemma \ref{Hermitenormalform} there exists a unimodular matrix $Q$ which transforms $\v$ and $\w$ into Hermite normal form, in particular $Q\v=(1, 0, \ldots, 0)$ and $Q\w=(q, p, 0, \ldots, 0)$ where $0\leq q<p$. If $q=0$, then $p=1$ since $\w$ is primitive, and hence $\Det_2(\v,\w)=1$. So assume that $q\geq 1$. In \cite[Section 3]{khurimatsumotoweinsteinyamada} an algorithm is presented that is based on the continued fraction decomposition of $p/q$, which produces a sequence of rod structures in $\Z^2$ connecting $(1,0)$ to $(q,p)$ such that each corner is admissible. We may then append zeros to each of the rod structures in this sequence, to obtain a sequence in $\Z^n$ that connects $(1, 0, \ldots, 0)$ to $(q, p, 0, \ldots, 0)$ with the same property. Applying $Q^{-1}$ then produces the desired sequence.
\end{proof}

This result was used in \cite{khurimatsumotoweinsteinyamada}, for $(4+1)$-dimensional spacetimes, to construct simply connected fill-ins for horizons. The simple connectivity of the fill-ins preserves the fundamental group of the DOC, and is not difficult to achieve since in this low dimensional setting admissible rod structures cannot contribute to the fundamental group. In higher dimensions this is not the case, and a more careful choice of rod structures is needed to achieve simply connected fill-ins. Moreover, since the boundary between the filled in region and the DOC now has a much larger fundamental group, there is a more complicated relation between the topologies of these regions.
In the last section, we will study the fundamental group of the compactified domain of outer communication.

\renewcommand{\B}{\mathcal{B}}

\section{The Model Map}

\label{modelmap}

In this section we construct a model map $\varphi_0\from\mathbb{R}^3\setminus\Gamma\to SL(n+1,\mathbb{R})/SO(n+1)$, which describes the singular behavior of the desired harmonic map near the axis $\Gamma$, as well as the asymptotics at infinity. The model map can be viewed as an approximate solution to the singular harmonic map problem near the axes and at infinity \cite{khuriweinsteinyamada2017,weinsteinMRL}. 
 We  define a model map as follows.
 
\begin{definition}\label{modeldef}
A map $\varphi_0\from\R ^3\setminus\Gamma\to  SL(n+1,\mathbb{R})/SO(n+1)$ is a \emph{model map} if
\begin{enumerate}
\item $\lvert \tau(\varphi_0)\rvert$ is bounded, where $ \tau$ denotes  the tension of $\varphi_0$, and
\item there is a positive function function $w\in C^{2}(\mathbb{R}^3)$ with ${\Delta w\leq-\lvert \tau(\varphi_0)\rvert}$ and $w\to 0$ at infinity.
\end{enumerate}
\end{definition}

It should be noted that if $\lvert \tau(\varphi_0)\rvert = O{(r^{-\alpha})}$ as $r\rightarrow\infty$, for some $\alpha>2$, then this is sufficient to satisfy condition $(2)$.  
In order to facilitate the construction of the model map, we will utilize the following parameterization of the target space. Namely, 
the target space is parameterized by $(F,\omega)$, where $F=(f_{ij})$ is a symmetric positive definite $n\times n$ matrix and $\omega=(\omega_i)$ is an $n$-tuple corresponding to the twist potentials. On each axis rod, the Dirichlet boundary data for $\omega_i$ is constant. These so called   \emph{potential constants} determine the angular momenta of the horizons, and do not vary between adjacent axis rods which are separated by a corner. In $(F,\omega)$ coordinates, the metric on the target space $SL(n+1,\mathbb{R})/SO(n+1)$ may be expressed as (see \cite{maison1979ehlers})
\begin{equation}\label{metricx}
\frac{1}{4}\frac{df^2}{f^2}+\frac{1}{4}f^{ij}f^{kl}df_{ik}df_{jl}+\frac{1}{2}\frac{f^{ij}d\omega_i d\omega_j}{f}=\frac{1}{4}[\Tr( F^{-1}dF)]^2+\frac{1}{4}\Tr(F^{-1}dFF^{-1}dF)\\+\frac{1}{2}\frac{d\omega^tF^{-1} d\omega}{f},
\end{equation}
where $f=\det F$ and $F^{-1}=(f^{ij})$ is the inverse matrix.
By setting 
\begin{equation}
H=F^{-1}\nabla F, \quad\quad G=f^{-1}F^{-1}({\nabla\omega})^2, \quad\quad K=f^{-1}F^{-1}\nabla\omega,
\end{equation}
it follow from \eqref{harmonic1} that the squared norm of the tension becomes
\begin{equation}\label{tension}
\lvert\tau\rvert^2=\frac{1}{4}[\Tr(\dv H+G)]^2+\frac{1}{4}\Tr[(\dv H+G)(\dv H+G)]+\frac{1}{2}f (\dv K)^t F (\dv K).
\end{equation}
It is clear from (\ref{tension}) that the tension norm is invariant under the transformation 
\begin{equation}
F\mapsto hFh^t \quad  \textrm{and}  \quad\omega\mapsto h\omega,
\end{equation} 
for any $h\in SL(n,\R)$. Note that $\det h=1$ is not required for this to hold when $\omega$ is constant, since $G$ and $K$ are then zero. The next result generalizes the model map construction from lower dimensions that was presented in \cite{khuriweinsteinyamada2017,khuriweinsteinyamada2018}.

\begin{figure}\label{domain}
	\includegraphics[width=10cm]{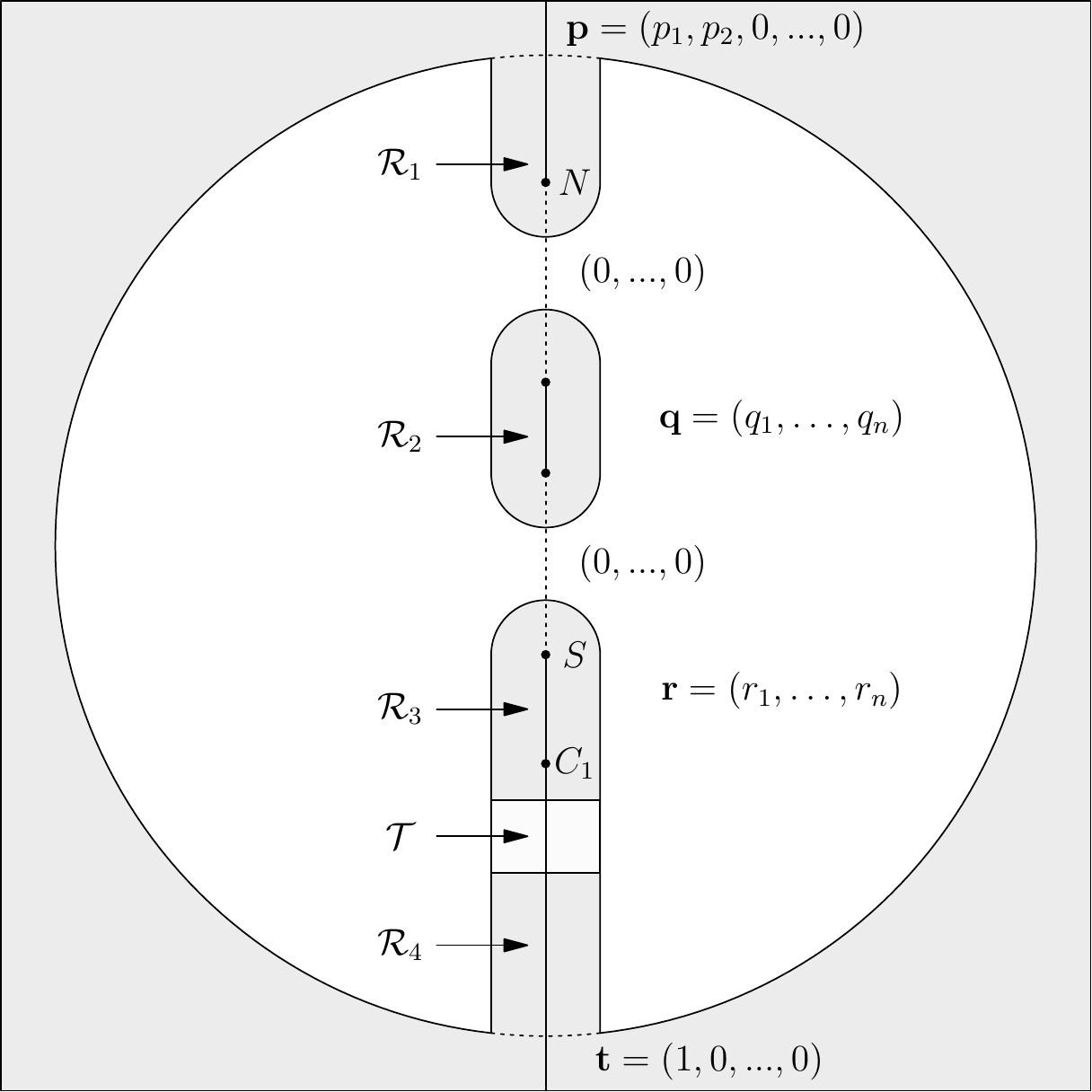}
	\caption{This diagram depicts the various regions used in the contruction of the model map. Axis rod structures are represented by $\mathbf{p}$, $\q$, $\r$, and $\t$, while horizon rods are indicated by dashed lines.}  \label{domain}
\end{figure}

\begin{lemma}\label{model}
For any admissible rod data set, with nondegenerate horizons, there exists a corresponding model map $\varphi_0\from\mathbb{R}^3\setminus\Gamma\to SL(n+1,\mathbb{R})/SO(n+1)$, for $n\geq 2$, having tension decay at infinity given by $|\tau|=O(r^{-5/2})$.
\end{lemma}
	

\begin{proof}
We first present a proof for the rod data set corresponding to two horizons and a single corner, as shown in Figure \ref{domain}. At the end of the proof, we will indicate the necessary adjustments for the general case. Observe that in the diagram there are four neighborhoods $\mathcal{R}_1$, $\mathcal{R}_2$, $\mathcal{R}_3$, and $\mathcal{R}_4$ associated with certain axis rods, having rod structures $\mathbf{p}$, $\q$, $\r$, and $\t$ respectively. The model map will be constructed separately in each of these regions. The following two harmonic functions on $\mathbb{R}^3\setminus\Gamma$ will play an important role in the construction
\begin{equation}\label{harmonicfunctn}	
u_a=\log(r_a-(z-a))=\log(2r_a \sin ^2(\theta_a/2)), \qquad v_a=\log(r_a+(z-a))=\log(2r_a \cos^2(\theta_a/2)),
\end{equation}
where $r_a=\sqrt{\rho^2+(z-a)^2}$ is the Euclidean distance from the point $z=a$ on the $z$-axis, and $\theta_a$ is the polar angle. 

Consider first the case in which the asymptotic end is modeled on $L(p,q)\times{T}^{n-2}$, where $0\leq q<p$. By applying Lemma \ref{Hermitenormalform} if necessary, it may be assumed without loss of generality that the rod structures on the semi-infinite rods are $\mathbf{p}=(p_1,p_2,0,\dots,0)$ with $p_2>0$, and $\mathbf{t}=(1,0,\dots,0)$. The model map outside of a large ball (corresponding to the shaded region outside of the circle in Figure \ref{domain}) and in the regions $\mathcal{R}_1$ and $\mathcal{R}_4$, may then be given by
\begin{equation}\label{fALE}
F_1=h\tilde{F_1}h^t,\quad \quad\quad \omega=h\tilde{\omega}(\theta),
\end{equation}
where $\tilde{\omega}$ is a function of $\theta=\theta_0$ alone described below and
\begin{equation}\label{aignq9uhg}
\tilde{F_1}= \diag{(e^{u_0-\log2},e^{v_0-\log2},1,...,1)},
\quad\quad\quad
h=\begin{pmatrix}
0 & \sqrt{p_2} &0\\
1/\sqrt{p_2} & -p_1/\sqrt{p_2} & 0\\
0 & 0 & \mathbf{I}_{n-2}
\end{pmatrix},
\end{equation}
with $\mathbf{I}_{n-2}$ representing the identity matrix. Notice that, up to multiplication by constants, $h^t$ sends $\mathbf{t}\mapsto \e_2$ and $\mathbf{p}\mapsto \e_1$. Thus, the matrix $F_1$ possesses the appropriate kernel at the semi-infinite rods to encode the given rod structures. Moreover, since $\varphi_0=(F_1,\omega)$ is obtained from the map $(\tilde F_1, \tilde{\omega})$ by applying an isometry to the target space, and $\tilde{F}_1$ arises from the canonical flat metric on $\mathbb{R}^4 \times T^{n-2}$, it follows that $\dv H=\dv F_1^{-1}\nabla F_1=0$. We may further choose $\tilde{\omega}(\theta)$ to be constant for $\theta \in [0,\epsilon]\cup [\pi-\epsilon,\pi]$, thus showing that $(F_1,\omega)$ is harmonic in $\mathcal{R}_1$ and $\mathcal{R}_4$. The constants are chosen to coincide with the prescribed potential constants on the axis rods. Within the remaining angular interval, $\tilde{\omega}(\theta)$ may be prescribed arbitrarily as long as it is smooth. In order to verify the decay of the tension for this map in the range $\theta \in  [\epsilon,\pi-\epsilon]$, observe that since $F_1=O(r)$, $f=O(r^{2})$, $|\nabla \omega| =O(r^{-1})$, and $\dv K=O(r^{-4})$ we have
\begin{equation}\label{decay4}
f (\dv K)^t F_1 (\dv K) =O(r^{-5}),\quad\quad\quad G=O(r^{-4}).
\end{equation}
Hence $\lvert \tau \rvert$ decays like $r^{-5/2}$, which is sufficient.
Similarly, in the case where the asymptotic end is modeled on $ S^2\times{T}^{n-1}$, we can without loss of generality assume that the rod structures on both the semi-infinite rods are $(1,0,\dots,0)$. The model map outside of the large ball and in the regions $\mathcal{R}_1$ and $\mathcal{R}_4$ is now given by
\begin{equation}\label{fAKK}
F_1= \diag{(e^{u},1,...,1)},\quad\quad\quad \omega=\omega(\theta),
\end{equation}
where $u=2\log \rho$ and $\omega$ is constant on $\theta \in [0,\epsilon]\cup [\pi-\epsilon,\pi]$. As before, the tension decays as $|\tau|=O(r^{-5/2})$ when $r\rightarrow\infty$.

Next consider the compact region $\mathcal{R}_2$  below the first horizon. The poles in this region are located at $z=a$ and $z=b$, $a<b$, and the rod structure is $\mathbf{q}=(q_1,q_2,\dots,q_n)$. The model map in this region is defined by
\begin{equation}
F_2=h_2\tilde{F_2} h_2^t,\quad\quad\quad \omega=c_2,
\end{equation}
where $\tilde{F_2}=\diag{(e^{u},1,...,1)}$, $u=u_a-u_b$, and
\begin{equation}
h_2=\left( \bigl[\mathbf{q},\mathbf{e}_2,\dots,\mathbf{e}_n\bigr]^t \right)^{-1}.
\end{equation}
The constant vector $c_2$ is chosen to agree with the prescribed potential constants on the rod.
As pointed out in the remark preceding the lemma, $\det h_2=1$ is not required here since $\omega$ is constant. It follows that the map $\varphi_0=(F_2,\omega)$  is harmonic in region $\mathcal{R}_2$.

Now we will deal with the regions $\mathcal{R}_3$, $\mathcal{R}_4$ and the transition region $\mathcal{T}$ between them.  Let the pole $S$ be at $z=s>0$ and the corner $C_1$ be at $z=0$. The rod structure above the corner $C_1$ is $\mathbf{r}=(r_1,\dots,r_n)$, and below the corner is $\mathbf{t}=(1,0,\ldots,0)$. Because of admissibility, we can without loss of generality assume that $r_2>0$. As above we set $\omega$ to be a constant $c_3$, agreeing with the prescribed potential constant on the rods, in the entire southern tubular neighborhoods $\mathcal{R}_3$ and $\mathcal{R}_4$.  Let
\begin{equation}
\tilde{F_3}=\diag{(	e^{u},e^{v},1,...,1)},\quad\quad u=(u_0-\log 2)-\lambda(z)(u_s -\log 2),\quad\quad v=v_0-\log 2,
\end{equation}	
where $\lambda=\lambda(z)$ is a smooth cut-off function which is 1 near $\mathcal{R}_3$ and 0 near $\mathcal{R}_4$. Define the map in region $\mathcal{R}_3$ by
\begin{equation}
F_3= h_3\tilde{F_3}h_3^t,\quad\quad\quad  \omega=c_3,
\end{equation}
where
\begin{equation}\label{airnhi2i}
h_3=\sqrt{p_2}\left(\bigl[\mathbf{r},\e_1, \e_3,\dots,\mathbf{e}_n\bigr]^t \right)^{-1}.
\end{equation}
We have already given the map in $\mathcal{R}_4$. In order to define the map in $\mathcal{T}$, set $h_3(z)$ to be a smooth curve of invertible $n\times n$ matrices which connects $h_3$ in \eqref{airnhi2i} to $h$ in \eqref{aignq9uhg}. Note that this is possible since both endpoint matrices have negative determinant, and that the curve may be chosen so that the second column of $\left(h_3(z)^t\right)^{-1}$ remains the constant vector $1/\sqrt{p_2} \e_1$. The map $F_3(z)=h_3(z)\tilde{F_3}(z)h_3^t(z)$ then identifies the correct rod structures, and agrees with the previously defined map on $\mathcal{R}_4$.
Since $\omega=c_3$, we have $G=K=0$ in $\mathcal{R}_3 \cup \mathcal{R}_4$. It remains to show that $\dv F_3^{-1}\nabla F_3$ is bounded on the transition region $\mathcal{T}$, since it vanishes on the complement. To see this, compute
\begin{align}\label{firstdvH}
\begin{split}
\dv F_3^{-1}\nabla F_3=&[\nabla(\tilde{F_3}h_3 ^t)^{-1}]\cdot (h_3 ^{-1}\nabla h_3) \tilde{F_3}h_3 ^t+(\tilde{F_3}h_3^t)^{-1} \div (h_3^{-1}\nabla h_3) \tilde{F_3}h_3^t\\
&+ (\tilde{F_3}h_3 ^t)^{-1} (h_3 ^{-1}\nabla h_3)\cdot \nabla (\tilde{F_3}h_3 ^t)+(\nabla h_3^{-t} )\cdot (\tilde{F}_3^{-1} \nabla\tilde{F}_3 )h_3^t\\
&+h_3^{-t}\operatorname{div}(\tilde{F}_3^{-1} \nabla\tilde{F}_3 )h_3^t
+h_3^{-t}(\tilde{F}_3^{-1} \nabla\tilde{F}_3)\cdot\nabla h_3^t
+\operatorname{div}(h_3^{-t}\nabla h_3).
\end{split}
\end{align}
Note that 
$|\nabla u|$ and $\partial_z v=1/r$ are clearly bounded in $\mathcal{T}$. Moreover,
the second row of $h_3^{-1}\nabla h_3$ vanishes, and this leads to the desired boundedness of $\dv F_3^{-1}\nabla F_3$.
Indeed, consider the first term on the right-hand side of \eqref{firstdvH}, namely
\begin{equation}
[\nabla(\tilde{F_3}h_3 ^t)^{-1}]\cdot (h_3 ^{-1}\nabla h_3) \tilde{F_3}h_3 ^t
=\left[\left(h_3^t \right)^{-1}\partial_z \tilde{F}_3^{-1} +\partial_z\left(h_3^t \right)^{-1} \cdot \tilde{F}_3^{-1}\right]
(h_3 ^{-1}\partial_z h_3) \tilde{F_3}h_3 ^t.
\end{equation}
The only potential difficulty in bounding this expression on $\mathcal{T}$ arises from the function $e^{-v}$, in $\tilde{F}_3^{-1}$ and $\partial_z \tilde{F}_3^{-1}$. However, since $h_3 ^{-1}\partial_z h_3$ has a vanishing second row, the products 
\begin{equation}
\tilde{F}_3^{-1} \cdot (h_3 ^{-1}\partial_z h_3),\quad\quad\quad \partial_z \tilde{F}_3^{-1} \cdot (h_3 ^{-1}\partial_z h_3),
\end{equation}
no longer contain $e^{-v}$ and the first term of \eqref{firstdvH} is controlled. The remaining terms may be handled analogously.
It follows that \eqref{firstdvH} is bounded, and hence the model map $\varphi_0=(F_3,\omega)$ has bounded tension in a tubular neighborhood of the two southern most rods. This treats the case in which the asymptotic end is modeled on $L(p,q)\times T^{n-2}$, and a similar procedure may be used in the case that the asymptotic end is modeled on $S^2 \times T^{n-1}$.

We will now address the multiple corner case. Any connected component of the axis consists of a consecutive sequence of axis rods. To construct the model map in a tubular neighborhood of such a component, first divide this region into neighborhoods centered at corners and transition regions between corners. The basic block consists of two such neighborhoods around adjacent corners $C_n$ and $C_s$, and the transition region $\mathcal T$ between them.
It suffices to illustrate the map construction in such blocks, as the full map may then be obtained by combining the individual pieces to handle any rod structure configuration.

Consider a basic block with rod structures $\mathbf{p}$, $\q$, and $\r$ on axis rods $\Gamma _1$, $\Gamma _2$, and $\Gamma_3$ respectively, moving from north to south.
Note that $\mathbf{p}$ and $\q$, as well as $\q$ and $\r$, must be linearly independent since the corners $C_n$ and $C_s$ are admissible. It follows that there is a collection of standard basis vectors $\{\e_{i_1},\dots,\e_{i_{n-2}}\}$ that complete $\{\mathbf{p},\q\}$ to a basis, and similarly for $\{\q,\r\}$. We may then form the matrices
\begin{equation}
h_{\mathbf{p},\q} = \left( \bigl[\mathbf{p}, \q, \e_{i_1},\dots,\e_{i_{n-2}} \bigr]^t \right)^{-1}, \quad\quad\quad h_{\r,\q} = \left( \bigl[\r, \q, \e_{j_1},\dots,\e_{j_{n-2}} \bigr]^t \right)^{-1}.
\end{equation}
Next define $F_0=\diag{(e^{u},e^{v},1,...,1)}$
where $u$ and $v$ are harmonic, with $e^u$ vanishing on $\Gamma_1$ and $\Gamma_3$, and $e^v$ vanishing on $\Gamma_2$. These functions may be given as the sum of logarithms of the form
\eqref{harmonicfunctn}. Then $F_0$ corresponds to the rod structures $\e_1$, $\e_2$, and $\e_1$ on $\Gamma_1$, $\Gamma_2$, and $\Gamma_3$ respectively.
Consider a smooth curve of invertible $n\times n$ matrices $h_{\mathbf{p}\mid \r,\q}(z)$ which agrees with $h_{\mathbf{p},\q}$ on $\Gamma_1$ and in a neighborhood of $C_n$, and transitions over $\mathcal{T}\subset\Gamma_2$ so that it agrees with $h_{\r,\q}$ on $\Gamma_3$ and in a neighborhood of $C_s$.
The existence of such a curve is possible since we may assume that the determinants of $h_{\mathbf{p},\q}$ and $h_{\r,\q}$ have the same sign by replacing $\r$ with $-\r$ if necessary. Moreover, the curve may be designed such that the second column of $\left(h_{\mathbf{p}\mid \r,\q}(z)^t\right)^{-1}$ is the constant vector $\q$. This implies that the second row of $h_{\mathbf{p}\mid \r,\q}^{-1} \nabla h_{\mathbf{p}\mid \r,\q}$ vanishes, so that with the help of \eqref{firstdvH} we find that $\dv F^{-1}\nabla F$ remains bounded along $\mathcal{T}$, where $F=h_{\mathbf{p}\mid \r,\q} F_0 h_{\mathbf{p}\mid \r,\q}^t$. The model map $\varphi_0=(F,\omega)$ on the basic block, with $\omega$ constant, then has bounded tension.

Lastly, it remains to treat the case of multiple blocks within an axis component. To accomplish this, take $u$ and $v$ harmonic so that $e^u$ and $e^v$ vanish in an alternating fashion on the string of axis rods. The diagonal matrix $F_0$ is then defined along the entire string. We will inductively construct the model map on basic block assemblies. As a demonstration of this, consider adding an additional rod $\Gamma_4$, with rod structure $\w$, to the sequence of three rods discussed above which we call basic block $\mathcal{B}_1$. We may view the $\Gamma_2$, $\Gamma_3$, $\Gamma_4$ string, with rod structures $\q$, $\r$, $\w$, as a basic block $\mathcal{B}_2$; 
the corner between the third and fourth rod will be denoted by $C_w$. 
The map has already been defined into a neighborhood of $\Gamma_3$, and may be extended into a neighborhood of $\Gamma_4$ as follows. Recall that the maps
\begin{equation}
F_1=h_{\mathbf{p}\mid \r, \q} F_0 h_{\mathbf{p}\mid \r,\q}^t,\quad\quad\quad\quad F_2=h_{\r, \q\mid \w}F_0 h_{\r,\q\mid \w}^t,
\end{equation}
are defined on the basic blocks $\mathcal{B}_1$ and $\mathcal{B}_2$ respectively, and identify the desired rod structures. However, they do not necessarily coincide on the overlap regions.
In order to remedy this situation, let $h_4(z)$ be a smooth curve of invertible $n\times n$ matrices connecting $h_{\r,\q}$ to $h_{\r,\w}$ with a transition over $\tilde{\mathcal{T}}\subset\Gamma_3$. This is possible since by replacing $\w$ with $-\w$ if necessary, we may assume that both endpoint matrices have determinants of the same sign. Moreover, this curve may be chosen such that the first column of $\left(h_4(z)^t \right)^{-1}$ remains the constant vector $\r$. Set $F=h_4(z) F_0 h_4(z)^t$ on $\Gamma_3$, and observe that this agrees with $F_1$ and $F_2$ near the corners $C_s$ and $C_w$, respectively, so that $F$ is naturally defined on all of $\mathcal{B}_1 \cup \mathcal{B}_2$. Since the first row of $h_4^{-1}\nabla h_4$ vanishes, we find with the aid of \eqref{firstdvH} that $\dv F^{-1}\nabla F$ remains bounded along $\Gamma_3$. The model map $\varphi_0=(F,\omega)$ on the two basic blocks, with $\omega$ constant, then has bounded tension. We may continue this process inductively to treat any number of consecutive axis rods. 
\end{proof}

\begin{remark}
\label{comp}
In \cite{khuriweinsteinyamada2018,khuriweinsteinyamada2017} an additional technical assumption on the rod structures, known as the compatibility condition, was used for the construction of the model map. The condition, which is not required for Lemma \ref{model}, states that given three adjacent rod structures with admissible corners, say $(m,n)$, $(p,q)$, and $(r,s)$, the following inequality must hold
\begin{equation}
\label{eq:compatibility}
mr(mq-np)(ps-rq)\leq 0 \text{ .}
\end{equation}
This turns out not to be a geometric condition, as it can always be achieved by a change of coordinates. To see this, first assume without loss of generality that the determinants $(mq-np)$ and $(ps-rq)$ are $1$, by possibly replacing $(p,q)$ or $(r,s)$ or both with the vector of the same length and opposite direction.
Note that this operation does not alter the isotropy subgroup prescribed by the rod structure.
Next apply the unimodular matrix 
\begin{equation}
A=\left( \begin{array}{cc} q & -p \\ -n & m \end{array}\right) 
\end{equation}
to obtain the rod structures $A\cdot \{(m,n), (p,q), (r,s)\}=\{(1,0), (0,1), (r',s')\}$, for some $r',s'\in \Z$. Then Equation \eqref{eq:compatibility} is clearly satisfied for the new set of rod structures. 
\end{remark}

\begin{remark}
Lemma \ref{model} and Remark \ref{comp} provide the proof of part $(a)$ from Theorem~\ref{theorem1}.
\end{remark}

\section{Horocyclic Coordinates and Energy Estimates}

\label{horo}
 
In this section we show how the energy estimates based on horocyclic coordinates can be generalized from the lower rank target space
setting that was treated in \cite[Section 6]{khuriweinsteinyamada2017}. The target space is now $SL(n+1,\R)/SO(n+1)$, which is a noncompact symmetric space of dimension $n(n+3)/2$ and rank $n$. For convenience we denote $G=SL(n+1,\R)$, $K=SO(n+1)$, and $\mathbf{X}=G/K$. The Iwasawa decomposition is given by $G=NAK$, where $A$ is the abelian group 
\begin{equation}
A=\{\diag(e^{\lambda_1},...,e^{\lambda_{n+1}})\mid \prod_{i=1} ^{n+1} e^{\lambda_i}=1 \},
\end{equation}
and $N$ is the nilpotent subgroup of upper triangular matrices with diagonal entries set to $1$. Thus, given $g\in G$ there are unique elements $m\in N$, $a\in A$, and $k\in K$ with $g=mak$, and the symmetric space $\mathbf{X}$ may be identified with the subgroup $NA$.
Denote $x_0=[Id]\in \mathbf{X}$ and note that the orbits $A\cdot x_0 =:\mathfrak{F}_{x_0}$ and $N\cdot x_0$ are respectively a maximal flat and a horocycle. The former is an $n$-dimensional totally geodesic submanifold with vanishing sectional curvature, and the latter is an $n(n+1)/2$-dimensional submanifold with the property that each flat which is asymptotic to the same Weyl chamber at infinity
has an orthogonal intersection with the horocycle in a single point. Furthermore, 
since each point $x \in \mathbf{X}$ may be uniquely expressed as $m a \cdot x_0$, 
the assignment $x \mapsto \mathfrak{F}_x=ma\cdot\mathfrak{F}_{x_0}$ yields a smooth foliation whose leaves are the flats $\{ m \cdot \mathfrak{F}_{x_0}\}_{m \in N}$; the flat $\mathfrak{F}_x$ orthogonally interects the horocycle $N\cdot x$ only at $x$.
In this manner, the pair $(a, m)$ gives rise to a horocyclic orthogonal coordinate system for $\mathbf{X}$.

A Euclidean coordinate system $r=(r_1,\ldots,r_n)$ may be introduced on 
$\mathfrak{F}_{x_0}$, and can then be pushed forward to each flat $m \cdot \mathcal{F}_{x_0}$ so that the horocyclic coordinates $(a, m)$ may be represented by $(r, m)$. Furthermore, each $r'$ defines a diffeomorphism (translation) $(r, m) \mapsto (r+r', m)$ that preserves the $m$-coordinates, and for each $m'\in N$ there is an isometry that preserves the $r$-coordinates $(r, m) \mapsto (r, m' m)$.
These $r$-translations map horocycles to horocylces, and therefore may be used to push forward a system of global coordinates
$\theta=(\theta^1,\ldots,\theta^{n(n+1)/2})$ on $N\cdot x_0\cong\mathbb{R}^{n(n+1)/2}$ to all horocycles. It follows that $(r,\theta)$ form a set of global coordinates on $\mathbf{X}$ in which the coordinate fields $\partial_{r_i}$ and $\partial_{\theta^j}$ are orthogonal, and such that the $G$-invariant Riemannian metric on $\mathbf{X}$ is expressed as
\begin{equation}
\mathbf{g}= dr^2+Q(d\theta,d\theta) = \sum_{i=1}^{n} dr_{i}^2  + \sum_{j,l=1}^{n(n+1)/2}Q_{jl}d \theta^{j} d\theta^{l},
\end{equation}
where the coefficients $Q_{jl}(r,\theta)$ are smooth functions. Moreover, the proof of \cite[Lemma 8]{khuriweinsteinyamada2017} generalizes in a direct manner to the current setting to yield the uniform bounds
\begin{equation}\label{aongag}
bQ(\xi,\xi)\leq\partial_{r_i}Q(\xi,\xi)\leq cQ(\xi,\xi),
\end{equation}
for all $i=1,\ldots,n$ and $\xi\in \R^{n(n+1)/2}$ where $0<b<c$. With the help of \eqref{aongag}, by expressing the harmonic map equations in the horocyclic parameterization we may establish energy bounds on compact subsets away from the axis. In particular, if $\varphi:\mathbb{R}^3 \setminus\Gamma\rightarrow \mathbf{X}$ is a harmonic map and $\Omega\subset\mathbb{R}^3 \setminus\Gamma$ is a bounded domain then the harmonic energy restricted to $\Omega$ satisfies
\begin{equation}
E_{\Omega}(\varphi)\leq \mathcal{C},
\end{equation}
where the constant $\mathcal{C}$ depends only on the maximum distance $\sup_{y\in\Omega}d_{\mathbf{X}}(\varphi(y),x_0)$.

\begin{definition}\label{asymptotic}
Two maps $\varphi_1,~\varphi_2\from\R ^3\setminus\Gamma\to \mathbf{X}$ are \emph{asymptotic} if there exists a constant $C$ such that $d_{\mathbf{X}}(\varphi_1,\varphi_2)\leq C$, and $d_{\mathbf{X}}(\varphi_1(y),\varphi_2(y))\to 0$ as $|y|\rightarrow \infty$.
\end{definition}

The distance between the model map and solutions to the harmonic map Dirichlet problem on an exhausting sequence of domains may be estimated via a maximum principle argument \cite{weinsteinMRL}, which is based on convexity of the distance function in the nonpositively curved target. This supremum bound together with the energy bound, allow for an application of standard elliptic theory to control all higher order derivatives. The sequence of harmonic maps on exhausting domains will then subconverge to the desired solution, for details see \cite[Sections 6 \& 7]{khuriweinsteinyamada2017}. We record this conclusion as the following result.

\begin{lemma}
Let $\varphi_0$ be a model map. Then there exists a unique harmonic map $\varphi:\mathbb{R}^3 \setminus \Gamma\rightarrow \mathbf{X}$ such that $\varphi$ is asymptotic to $\varphi_0$.
\end{lemma}

This lemma establishes part $(b)$ of Theorem~\ref{theorem1}. Since $\varphi$ is asymptotic to $\varphi_0$, it can be shown in the same way as \cite[Theorem 11]{khuriweinsteinyamada2017}, that the two maps respect the same rod data set. Furthermore, part $(c)$ of Theorem \ref{theorem1} may be established analogously to \cite[Section 8]{khuriweinsteinyamada2017}. This completes the proof of Theorem \ref{theorem1}.



\section{Plumbing and Topology of the Domain of Outer Communication}

\label{plumbing}


	\renewcommand{\B}{B}
	\renewcommand{\D}{D}
	\newcommand{\m}{\vec{m}}
	\newcommandx{\setoptions}[5][1=ERROR, 2=ERROR, 3=ERROR, 4=ERROR, 5=ERROR]{
		\def\optionA{#1}
		\def\optionB{#2}
		\def\optionC{#3}
		\def\optionD{#4}
		\def\optionE{#5}
		}
	\newcommand{\textoption}[1]{
		\ifnum#1=1
			\optionA \fi
		\ifnum#1=2
			\optionB \fi
		\ifnum#1=3
			\optionC \fi
		\ifnum#1=4
			\optionD \fi
		\ifnum#1=5
			\optionE \fi
		}

There are two methods that can be used to characterize the domain of outer communication. One method consists of filling in horizons and cross-sections in the asymptotic end to obtain a simply connected compact manifold.
In the next section we use this method for spatial dimensions $4$, $5$, and $6$, where a complete list of possible topologies is available.  The other approach involves breaking up the domain of outer communication into simpler pieces, and then classifying the individual components. This is the method of plumbing constructions which will be discussed in the current section, and will yield the proof of Theorem \ref{main:DOC}. Throughout this section we will assume that $n\geq 3$.

In Theorem \ref{main:DOC} the domain of outer communication is broken up into components determined by the number of corners that they contain.
The pieces which contain no corners are either the asymptotic end $M_{end}^{n+2}$, or a piece which is homeomorphic to $[0,1]\times \D^2\times T^{n-1}$ which we denote by $C^{n+2}_k$.
When a piece contains a single corner, the admissibility condition may be used to show that it is the product of a ball with a torus $\B^4\times T^{n-2}$. This part of the analysis is identical to the (spatial) $4$-dimensional case that is covered in \cite[Theorem~1]{khurimatsumotoweinsteinyamada}.
However, a significant difference occurs in higher dimensions when analyzing components that contain at least two corners.
A component with exactly two corners will turn out to be the product of a torus $T^{n-3}$ with a disk bundle over a $3$-manifold, rather than a $2$-sphere. Moreover, for components with more than two corners, we will have to define a generalization of plumbing where the fibers and base space are not of the same dimension.

\begin{theorem}
Let $M^{n+2}$ be a simple $T^n$-manifold, and consider a neighborhood $N^2$ in the orbit space of a portion of the axis with two corners and no horizon rods. The total space over $N^2$ is homeomorphic to $\xi\times T^{n-3}$, where the action of $T^n\iso T^3 \times T^{n-3}$ acts componentwise.
Here $\xi$ is a $\D^2$-bundle
over $X\in \{S^3, L(p,q), S^1\times S^2\}$.
The topologies of $X$ and $\xi$ may be read off from the Hermite normal form of the rod structures.
\label{thrm:Disk Bundles}
\end{theorem}

\begin{proof}
The rod diagram of $N^2$ has three axis rods separated by two admissible corners. Using Remark \ref{rmk:simple} we can, without changing the topology, transform our rod structures into the form of Equation \eqref{eq:qrp}, where the last $n-3$ entries of each rod structure are zero. The last $n-3$ Killing fields then do not vanish over $N^2$, and hence the total space is a product manifold $\xi\times T^{n-3}$, where the $T^n$-action splits naturally into $T^3$ acting on $\xi$ and $T^{n-3}$ acting on itself. Here $\xi$ denotes the manifold represented by the rod diagram $\{(1,0,0), (0,1,0), (q,r,p)\}$ with $0\leq q<p$, $0\leq r<p$, and $\mathrm{gcd}\{q,p\}=1$ if the vectors are linearly independent. In the case that they are linearly dependent, we instead have $q=1$, $p=0$, and $r\in\mathbb{Z}$.

The middle axis rod, where the second Killing field vanishes, is a deformation retract of the space $\xi$. This rod represents a closed manifold $X\in \{S^3, L(p,q), S^1\times S^2\}$. Fibers over this space correspond to rays extending out from the middle axis rod, see Figure \ref{fig:plumbing2}. Each point in the interior of the middle axis rod corresponds to an entire $T^2$, while a ray terminating at that point corresponds to $\D^2\times T^2$. Moreover, each of the two corners corresponds to an $S^1$ in the base space $X$, while the adjacent axis rods correspond to $\D^2\times S^1$. It follows that $\xi$ has the structure of a $\D^2$-bundle over $X$.


To determine the topology of $X$ and $\xi$, we look at the rod structures. If they are linearly dependent, then by admissibility the rod structures must be $\{(1,0,0), (0, 1, 0), (1, r, 0)\}$. There is then a free $S^1$ action, and after factoring this out, it remains to analyze the 4-dimensional disk bundle generated by the diagram with rod structures $\{(1,0), (0, 1), (1, r)\}$. The base space of this latter disc bundle is $S^2$, and its zero-section self-intersection number, or equivalently the characteristic number of its Euler class is $r$, see~\cite{khurimatsumotoweinsteinyamada}. Moreover, we have $X=S^1\times S^2$.

If the rod structures $\{(1,0,0), (0, 1, 0), (q, r, p)\}$ are linearly independent, the base space $X=L(p,q)$. Recall that $L(1,q)=S^3$ for all $q$. The number of distinct disk bundles, or equivalently $SO(2)$-bundles, over $X$ is determined by the homotopy classes of maps $[X,\CP^\infty]$. Moreover, the classifying space $BS^1=\CP^\infty$ is an Eilenberg-Maclane space of type $K(\Z, 2)$, so the homotopy classes of based maps from $X$ to $K(\Z,2)$ is in bijection with $H^2(X; \Z)\iso\Z_p$. The element of this cohomology group which corresponds to a specific bundle $\xi$ is called the \emph{Euler class} $e(\xi)$. 

By uniqueness of the Hermite normal form, the $r\in \Z_p\iso H^2(L(p,q);\Z)$ in the rod structure is uniquely determined for each equivariant homeomorphism class of $\xi$. Conversely, for each class in $H^2(L(p,q);\Z)$ there is a unique disk bundle over $L(p,q)$. Each of these disk bundles admits an effective $T^3$ action, with $T^1$ acting on the fibers, and a $T^2$ acting on the base $L(p,q)$. Thus, to each of these disk bundles corresponds a rod diagram with three axis rods and two admissible corners. This gives a one-to-one correspondence between integers $r\in[0,p)$ and $e(\xi)\in H^2(L(p,q), \Z)$. Furthermore, for the trivial disk bundle $L(p,q)\times \D^2$ both $r=0$ and $e(\xi)=0$. To see this, note that the quotient of $L(p,q)$ by its $T^2$-action can be represented as an interval where the $(1,0)$ and the $(q,p)$ circles degenerate at the end points. Similarly, the quotient of $\D^2$ by $S^1$ can be represented by a half open interval where the circle degenerates at the one end point. Taking the product of these two spaces produces the rod diagram $\{ (1,0,0), (0, 1, 0), (q, 0, p)\}$, from which we deduce that $r=0$.
\end{proof}

The above theorem shows that the total space over a neighborhood of three consecutive axis rod structures $\{\u, \v, \w\}$, satisfying the admissiblity condition, is $\xi\times T^{n-3}$ where $\xi$ is a disk bundle over either a lens space or a ring. Observe that there is a subtorus $T^3$ which leaves the slices $\xi\times \{\boldsymbol\varphi\}\in \xi\times T^{n-3}$ invariant, and is spanned by the rod structures $\{\u, \v, \w\}\sub\Z^n$ as follows
\begin{equation}
\label{eq:T3}
T^3 \iso \operatorname{span}_\R \{\u, \v, \w\} /\Z^n 
\sub \R^n / \Z^n
\iso T^n .
\end{equation}
Although $\{\u, \v, \w\}$ may not necessarily be a primitive set, this can be rectified by employing an integral version of the Gram-Schmidt process, which will lead to the formulation of generalized plumbing.

\begin{lemma}
\label{thrm:preplumb}
Let $\{\u, \v, \w\}\sub \Z^n$ be a consecutive sequence of rod structures satisfying the admissibility condition, and with a neighborhood that lifts to $\xi\times T^{n-3}$ in the total space. If $\xi$ is a $D^2$-bundle over $L(p,q)$, $0\leq q<p$ with Euler class determined by $r\in[0,p)$, then there exists a unique primitive vector $\p\in \Z^n$ satisfying
\begin{equation}
\label{eq:preplumb}
\w = q \u + r \v + p \p .
\end{equation}
Furthermore, $\{\u, \v, \p\}\sub \Z^n$ forms a primitive set. In addition, if $\xi$ is a $D^2$-bundle over $S^1\times S^2$, then Equation \eqref{eq:preplumb} is satisfied with $\p=0$.
\end{lemma}

\begin{proof}
First consider the case in which $\xi$ is a $D^2$-bundle over $L(p,q)$, $0\leq q<p$ with Euler class determined by $r\in[0,p)$. Let $Q$ be the unimodular matrix that transforms $\{\u, \v, \w\}$ into Hermite normal form, that is $Q\u=\e_1$, $Q\v = \e_2$, and $Q \w= q \e_1 +  r \e_2 + p\e_3$. We may then set $\p=Q^{-1}\e_3$ and observe that \eqref{eq:preplumb} is satisfied. Since the Hermite normal form is unique, and $p\neq 0$, it is clear that $\p\in\mathbb{Z}^n$ is the unique solution to the equation. Furthermore, since $Q^{-1}$ is unimodular and $\e_3$ is a primitive vector we find that $\p$ is primitive as well. Next note that $\{\u, \v, \p\}$ is a primitive set if and only if $\Det_3(\u, \v, \p)=1$. Moreover, by multilinearity of the determinant together with Equation \eqref{eq:preplumb}, it follows that
\begin{equation}
\Det_3(\u,\v,\p) = p^{-1} \Det_3(\u, \v, \w) = p^{-1}\Det_3(\e_1, \e_2, q\e_1 +r\e_2+p\e_3)=1,
\end{equation}
where the second equality follows from the coordinate invariance of $\Det_3$. Lastly, if $\xi$ is a $D^2$-bundle over $S^1\times S^2$, then $q=1$ and $p=0$ so that \eqref{eq:preplumb} is satisfied with $\p=0$.
\end{proof}

We will now consider portions of the axis having more than two consecutive corners in a simple $T^n$-manifold. The total space over neighborhoods of these regions of the axis, with $l+1$ corners, will be shown to consist of $l$ disk bundle-torus products that are glued together in a fashion that may be viewed as a generalization of the linear plumbing construction. This higher dimensional plumbing, which we will refer to as \emph{toric plumbing}, is not a straightforward generalization of 4-dimensional procedure due to the various ways that the extra toroidal dimensions may be conjoined. For each pair of neighboring disk bundles we will define a \emph{plumbing vector}, which distinguishes the different ways that the two disk bundles can be plumbed together. Figure \ref{fig:badplumb} provides examples of the same two disk bundles being plumbed together in different ways to form non-homeomorphic total spaces.

Consider a section of the axis rod, having admissible corners, with rod structures $\{\v_1, \ldots, \v_{l+2}\}$. From Theorem \ref{thrm:Disk Bundles}, a neighborhood of each consecutive triple of rod structures $\{\v_{i}, \v_{i+1}, \v_{i+2}\}$ lifts to the total space as a product $\txi_i \iso \xi_i \times T^{n-3} \sub M^{n+2}$, where $\xi_i$ is a disk bundle with Euler class determined by $r_i$ over either $L(p_i, q_i)$, or $S^1\times S^2$ if $p_i=0$. With the aid of a unimodular transformation matrix $Q$, we can arrange the rod structures into Hermite normal form $\{\w_1, \ldots, \w_{l+2}\}$ so that $Q\v_i = \w_i$. Recall that the $\w_i$ are uniquely determined, although $Q$ may not have this property. By Remark \ref{rmk:simple}, the first three elements are given by $\w_1=\e_1$, $\w_2=\e_2$, and $\w_3 = (q_1, r_1, p_1, 0, \ldots, 0)$. For each $i$ such that $p_i\ne 0$, Lemma~\ref{thrm:preplumb} ensures the existence of a unique primitive vector $\p_i \in\mathbb{Z}^n$ satisfying
\begin{equation}
\label{eq:plumb}
\w_{i+2} = q_i \w_i + r_i \w_{i+1} + p_i \p_i.
\end{equation}
When $p_i=0$ we define $\p_i=\boldsymbol 0$, and~\eqref{eq:plumb} is trivially satisfied.

\begin{definition}
\label{def:plumb}
The vectors $\p_i$ satisfying~\eqref{eq:plumb} are referred to as \emph{plumbing vectors}.
\end{definition}

\begin{remark}
\label{rmk:geometric}
\sloppy
If $\bar{Q}$ is a unimodular matrix, then $\{\v_1, \ldots, \v_{l+2}\}$ and $\{\bar{Q} \v_1, \ldots, \bar{Q} \v_{l+2}\}$ have the same Hermite normal form and thus the same plumbing vectors. Therefore, plumbing vectors do not depend on the choice of coordinates, but rather depend only on the toric structure of the total space.
\end{remark}

While the set of plumbing vectors is uniquely determined by a set of rod structures, they are not uniquely determined by the topologies of $\xi_i$.
In Figure \ref{fig:badplumb}, we present two pairs of examples in which the same disk bundles are being plumbed with different plumbing vectors. From Remark \ref{rmk:geometric} we know that the total spaces will have different toric structures, and will not simply differ by a change of coordinates. Furthermore, in these examples the boundaries of the total spaces have different fundamental groups. Thus, plumbing vectors can affect the topology of the total space.

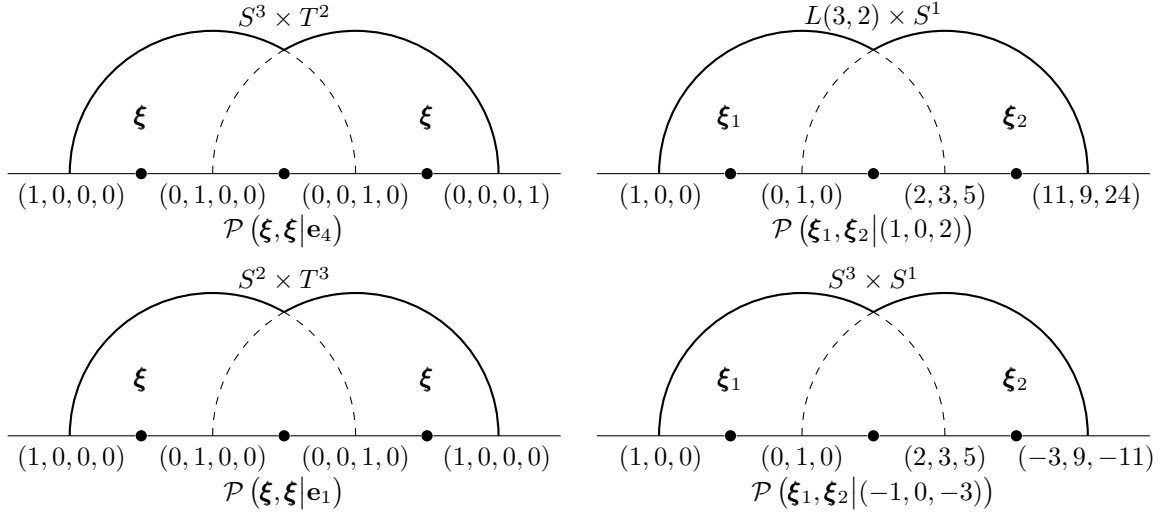
\begin{figure}\centering
\begin{tikzpicture}
\setoptions[(1,0,0,0)][(0,1,0,0)][(0,0,1,0)][(0,0,0,1)]

\tikzmath{\stepsize = 1.9; \raynum = 10;}

\foreach \n in {1,2,3}
\node at ( \n * \stepsize, 0) [circle,fill,inner sep=1.5pt] (\n) {};
\node at (0,0) (0) {};
\node at (4*\stepsize, 0) (4) {};

\foreach \n/\m in {0/1,1/2,2/3,3/4}
\draw (\n) -- node[pos =.5, below] {$\textoption{\m}$} (\m);

\draw[thick] (\stepsize*.5, 0) arc (180:60:\stepsize);
\draw[thick] (\stepsize*3.5, 0) arc (0:120:\stepsize);

\draw[dashed] (\stepsize*2.5, 0) arc (0:60:\stepsize);
\draw[dashed] (\stepsize*1.5, 0) arc (180:120:\stepsize);

\node at (2*\stepsize, -.4*\stepsize) {$\plumb{\txi, \txi\big| \e_4}$};
\node at (2*\stepsize, 1.1*\stepsize) {$S^3\times T^2$};
\node at (\stepsize, .4*\stepsize) {$\txi$};
\node at (3*\stepsize, .4*\stepsize) {$\txi$};
\end{tikzpicture}%
\begin{tikzpicture}
\setoptions[(1,0,0)][(0,1,0)][(2,3,5)][(11,9,24)]

\tikzmath{\stepsize = 1.9; \raynum = 10;}

\foreach \n in {1,2,3}
\node at ( \n * \stepsize, 0) [circle,fill,inner sep=1.5pt] (\n) {};
\node at (0,0) (0) {};
\node at (4*\stepsize, 0) (4) {};

\foreach \n/\m in {0/1,1/2,2/3,3/4}
\draw (\n) -- node[pos =.5, below] {$\textoption{\m}$} (\m);

\draw[thick] (\stepsize*.5, 0) arc (180:60:\stepsize);
\draw[thick] (\stepsize*3.5, 0) arc (0:120:\stepsize);

\draw[dashed] (\stepsize*2.5, 0) arc (0:60:\stepsize);
\draw[dashed] (\stepsize*1.5, 0) arc (180:120:\stepsize);

\node at (2*\stepsize, -.4*\stepsize) {$\plumb{\txi_1, \txi_2\big| (1,0,2)}$};

\node at (2*\stepsize, 1.1*\stepsize) {$L(3,2)\times S^1$};
\node at (\stepsize, .4*\stepsize) {$\txi_1$};
\node at (3*\stepsize, .4*\stepsize) {$\txi_2$};
\end{tikzpicture}

\begin{tikzpicture}
\setoptions[(1,0,0,0)][(0,1,0,0)][(0,0,1,0)][(1,0,0,0)]

\tikzmath{\stepsize = 1.9; \raynum = 10;}

\foreach \n in {1,2,3}
\node at ( \n * \stepsize, 0) [circle,fill,inner sep=1.5pt] (\n) {};
\node at (0,0) (0) {};
\node at (4*\stepsize, 0) (4) {};

\foreach \n/\m in {0/1,1/2,2/3,3/4}
\draw (\n) -- node[pos =.5, below] {$\textoption{\m}$} (\m);

\draw[thick] (\stepsize*.5, 0) arc (180:60:\stepsize);
\draw[thick] (\stepsize*3.5, 0) arc (0:120:\stepsize);

\draw[dashed] (\stepsize*2.5, 0) arc (0:60:\stepsize);
\draw[dashed] (\stepsize*1.5, 0) arc (180:120:\stepsize);

\node at (2*\stepsize, -.4*\stepsize) {$\plumb{\txi, \txi\big| \e_1}$};
\node at (2*\stepsize, 1.1*\stepsize) {$S^2\times T^3$};
\node at (\stepsize, .4*\stepsize) {$\txi$};
\node at (3*\stepsize, .4*\stepsize) {$\txi$};
\end{tikzpicture}%
\begin{tikzpicture}
\setoptions[(1,0,0)][(0,1,0)][(2,3,5)][(-3,9,-11)]

\tikzmath{\stepsize = 1.9; \raynum = 10;}

\foreach \n in {1,2,3}
\node at ( \n * \stepsize, 0) [circle,fill,inner sep=1.5pt] (\n) {};
\node at (0,0) (0) {};
\node at (4*\stepsize, 0) (4) {};

\foreach \n/\m in {0/1,1/2,2/3,3/4}
\draw (\n) -- node[pos =.5, below] {$\textoption{\m}$} (\m);

\draw[thick] (\stepsize*.5, 0) arc (180:60:\stepsize);
\draw[thick] (\stepsize*3.5, 0) arc (0:120:\stepsize);

\draw[dashed] (\stepsize*2.5, 0) arc (0:60:\stepsize);
\draw[dashed] (\stepsize*1.5, 0) arc (180:120:\stepsize);

\node at (2*\stepsize, -.4*\stepsize) {$\plumb{\txi_1, \txi_2\big| (-1,0,-3)}$};
\node at (2*\stepsize, 1.1*\stepsize) {$S^3\times S^1$};
\node at (\stepsize, .4*\stepsize) {$\txi_1$};
\node at (3*\stepsize, .4*\stepsize) {$\txi_2$};
\end{tikzpicture}
\caption{The left two examples represent different toric plumbings of the trivial bundle $\txi=S^3 \times \D^2\times S^1$ with itself. In the top left example the plumbing vector is $\p_2=\e_4$, while in the bottom left example the plumbing vector is $\p_2=\e_1$. The right two examples are different toric plumbings of $\txi_1$ over $L(5,2)$ with Euler class determined by $3$, and $\txi_2$ over $L(7,3)$ with Euler class determined by $2$. The plumbing vector for the top right example is $\p_2=(1,0,2)$, while the plumbing vector for the bottom right example is $\p_2 =(-1, 0, -3)$. We can see that for each pair the topology and toric structure of the total space is different, as a consequence of having different plumbing vectors. The notation $\plumb{\txi_1,\txi_2,\p}$ refers to the toric plumbing of $\txi_1$ and $\txi_2$ with plumbing vector $\p$, as given in Definition~\ref{def:plumbing}.}
\label{fig:badplumb}
\end{figure}

\begin{subequations}
\label{eq:plumbrelations}
Plumbing vectors satisfy a number of relations, the first of which is the collection of recursion equations that are used in the definition
\begin{equation}
\label{eq:recursion}
\begin{gathered}
\w_1=\e_1 \text{, }	\w_2 = \e_2 \text{, }\\
\w_{i+2} =  q_i \w_i + r_i \w_{i+1} + p_i \p_i \text{ if } p_i\not=0 \text{, and }\\
\p_i=0 \text{ if } p_i=0 \text{, }
\end{gathered}
\end{equation}
for $i=1, \ldots, l$.  The next two conditions arise from are admissibility of the corners, and primitivity of the triples containing the plumbing vector. More precisely, adjacent rods $\{\w_{i+1}, \w_{i+2}\}$ are assumed to have an admissible corner, that is $\Det_2(\w_{i+1}, \w_{i+2})=1$. By using the recursion relations and the multilinearity of determinants, this can be re-expressed as
\begin{equation}
\label{eq:admissible w}
\Det_2 ( \w_{i+1}, q_i \w_i + p_i \p_i) = 1.
\end{equation}
Furthermore, the primitivity condition that is guaranteed by Lemma~\ref{thrm:preplumb} asserts that
\begin{equation}
\label{eq:primitivity}
\Det_3(\w_i, \w_{i+1}, \p_i)=1,
\end{equation}
when $\p_i\not=0$. If $\p_i=0$ then this condition does not apply.
Finally, we obtain two conditions from the fact that $\{\w_0, \ldots, \w_{l+2}\}$ is in Hermite normal form. The first describes conditions under which certain entiees must vanish. That is, if $\pc_{i j}=0$ for all $j\geq m$ and $1\leq i<k$, where $\p_i = (\pc_{i 1}, \ldots, \pc_{i n})$, then
\begin{equation}
\label{eq:zeros}
\pc_{k j}=0\quad\text{ for all }\quad j>m.
\end{equation}
The second condition indirectly restricts the size of certain components in the plumbing vectors.
Write $\w_i=(w_{i1},\ldots, w_{in})$, and denote the last nonzero entry of $\p_k$ by $\pc_{k m_k}$. If $\pc_{i m_k}=0$ for all $1\leq i<k$, then $w_{(k+2)m_k}$ is a pivot in the Hermite normal form so that
\begin{equation}
\label{eq:pivot}
0\leq w_{(k+2)j}< w_{(k+2)m_k} \quad\text{ for all }\quad j<m_k.
\end{equation}
These relations will be collectively referred to as the \emph{plumbing relations}.
\end{subequations}

The first plumbing vector $\p_1$ takes a simple form in all cases, depending only on whether $p_1$ vanishes.
Namely, if the base space of $\xi_1$ is $S^1\times S^2$ then $p_1=0$, and we have $\p_1=0$.
If $p_1\not=0$ then note that Remark \ref{rmk:simple} implies $\w_3=(q_1,r_1,p_1,0,\ldots,0)$. This immediately shows that
$\p_1=\e_3$ solves Equation \eqref{eq:recursion}, and by uniqueness of plumbing vectors it follows that $\p_1$ must take this
form. In what follows, since $\p_1$ is determined only by the topology of $\xi_1$ and not by plumbing information, we do not include it when describing the toric plumbing of $\xi_1$ and $\xi_2$. Thus, only $l-1$ plumbing vectors are needed to describe the gluing for a string of $l+2$ rod structures.

\begin{prop}
\label{thrm:wp}
There is a one-to-one correspondence between collections of admissible rod structures $\{\w_1, \ldots, \w_{l+2}\}\sub \Z^n$ in Hermite normal form, and collections of bundles $\{\txi_1, \ldots, \txi_l\}$
paired with a set of primitive vectors $\{\p_2, \ldots, \p_l\}\sub \Z^n$ satisfying Equations \eqref{eq:plumbrelations}.
\end{prop}

\begin{proof}
Let $\{\w_1, \ldots, \w_{l+2}\}\sub \Z^n$ be a collection of admissible rod structures in Hermite normal. The proof of Theorem \ref{thrm:Disk Bundles} shows that from each successive triple $\{\w_i, \w_{i+1}, \w_{i+2}\}$, there is a unique bundle $\txi_i$ which is the lift of a (orbit space) neighborhood of these three rods to the total space $M^{n+2}$. The rod structures also give the integers $q_i$, $r_i$, and $p_i$ used in Definition~\ref{def:plumb} to obtain the plumbing vectors $\p_i$. By construction, together with the admissiblity condition, these vectors satisfy the full set of plumbing relations \eqref{eq:plumbrelations}.

Conversely, let $\{\txi_1, \ldots, \txi_l\}$ be a collection of bundles
and let $\{\p_2, \ldots, \p_l\}\sub \Z^n$ be a collection of vectors satisfying Equations \eqref{eq:plumbrelations}. According to the discussion preceding this proposition, we may append to this list $\p_1=0$ if the base of $\txi_1$ is $S^1 \times S^2$, or $\p_1=\e_3$ if the base of $\txi_1$ is a lens space. Equation \eqref{eq:recursion} then uniquely determines
the rod structures $\{\w_1, \ldots, \w_{l+2}\}$, since the integers $q_i$, $r_i$, and $p_i$ are uniquely defined by each $\txi_i$ as in the proof of Theorem \ref{thrm:Disk Bundles}.
By hypothesis, the vectors $\{\w_1, \ldots, \w_{l+2}\}$ satisfy \eqref{eq:admissible w} which can be rewritten as $\Det_2 (\w_{i+1}, \w_{i+2})=1$, thus establishing admissibility. Lastly, we note that
Equations \eqref{eq:recursion} and \eqref{eq:pivot} imply that the matrix composed of column vectors $\w_i$ satisfies the conditions of Lemma \ref{thrm:Hermite}. Thus, the collection of rod structures is in Hermite normal form.
\end{proof}

\begin{definition}
\label{def:plumbing}
Let $\txi_i\iso  \xi_i \times T^{n-3}$, $i=1,\dots,l$ where each $\xi_i$ is a $\D^2$-bundle over either a 3-dimensional lens space or $S^1\times S^2$, and let $\{ \p_2, \ldots, \p_l\}\sub \Z^n$ be a collection of primitive vectors satisfying the plumbing relations \eqref{eq:plumbrelations}.
We define the \emph{toric plumbing} of $\txi_1, \ldots, \txi_l$ along the plumbing vectors $\p_2, \ldots, \p_l$ to be the $(n+2)$-dimensional simple $T^n$-manifold given by rod structures $\{\w_1, \ldots, \w_l\}$, where the $\w_i$ are determined by Equations \eqref{eq:recursion}. This simple $T^n$-manifold is denoted by $\plumb{\txi_1, \ldots, \txi_l \big| \p_2, \ldots, \p_l}$.
\end{definition}

\begin{figure}\centering\centering
\begin{tikzpicture}

\tikzmath{\stepsize = 3; \raynum = 10;}
\setoptions[\w_1][\w_2][\w_3][\w_4]

\foreach \n in {1,2,3}
\node at ( \n * \stepsize, 0) [circle,fill,inner sep=1.5pt] (\n) {};
\node at (0,0) (0) {};
\node at (4*\stepsize, 0) (4) {};

\foreach \n/\m in {0/1,1/2,2/3,3/4}
\draw (\n) -- node[pos =.5, below] {$\textoption{\m}$} (\m);

\draw (\stepsize*.5, 0) arc (180:0:\stepsize);
\draw (\stepsize*3.5, 0) arc (0:180:\stepsize);

\foreach \n in {-\raynum,...,\raynum}
\draw[red] (1.5*\stepsize, 0)+(.5*\n*\stepsize/\raynum, 0) -- +(90 - \n*90/\raynum:\stepsize) coordinate(\n);

\foreach \m in {4,...,\raynum}
\draw[red] (\m) to [out=20*10/4-20*\m/4, in=90+90*\m/6-90*4/6] (3.5*\stepsize - .5*\stepsize*\m/7 + .5*\stepsize*3/7, 0);

\foreach \n in {-\raynum,...,\raynum}
\draw[blue] (2.5*\stepsize, 0)+(.5*\n*\stepsize/\raynum, 0) -- +(90 - \n*90/\raynum:\stepsize) coordinate(\n);

\foreach \m in {-4,...,-\raynum}
\draw[blue] (\m) to [out=180 - 20*10/4-20*\m/4, in=10*90/7+\m*90/7] (.5*\stepsize - .5*\stepsize*\m/7 - .5*\stepsize*3/7, 0);

\end{tikzpicture}


\caption{In the figure above we have $\w_1=\e_1$, $\w_2 = \e_2$, $\w_3 = (q_1, r_1, p_1)$, and $\w_4 =q_2 \w_2 + r_2 \w_3 + p_2 \p_2$ in accordance with Equation \eqref{eq:recursion}. The diagram shows a toric plumbing of two disk bundle-torus products $\txi_1$ and $\txi_2$ over lens spaces $L(p_1, q_1)$ and $L(p_2, q_2)$, along plumbing vector $\p_2$. The fibers of $\txi_1$ are given by rays emanating from $\w_2$, while the fibers of $\txi_2$ are given by rays emanating from $\w_3$. Note that in the overlap, the fibers and sections switch roles between $\txi_1$ and $\txi_2$.
}
\label{fig:plumbing2}
\end{figure}
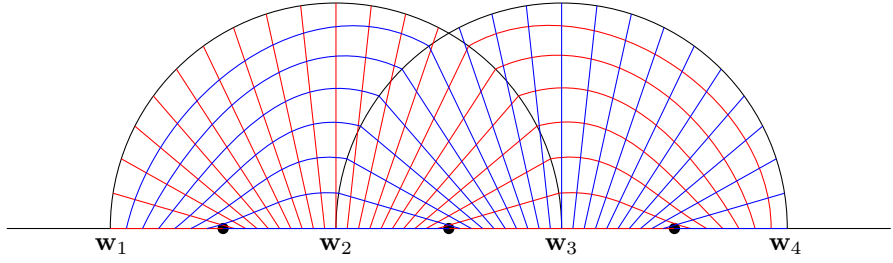

Toric plumbing may be considered as a generalization of standard equivariant plumbing. In the latter construction the base and the fiber have the same dimension, while in the former they do not. In order to elucidate the similarity between the two notions of plumbing, we restrict attention to $n=3$ and consider a simple $T^3$-manifold $\plumb{\txi_1, \txi_2 \big | \p_2}$. First note that this represents a gluing of $\txi_1$ and $\txi_2$. Indeed, the inclusion $\txi_1\into \plumb{\txi_1, \txi_2 \big | \p_2}$ is manifested by the fact that $\{\w_1, \w_2, \w_3\}$ gives the canonical (Hermite normal form) rod diagram for $\txi_1$. Furthermore, the inclusion of $\txi_2$ may be observed by applying a unimodular transformation $Q$ which sends $\w_2$ to $\e_1$, $\w_3$ to $\e_2$, and sends $\p_2$ to $\e_3$ if $\p_2\not=0$, to obtain the rod structures $\{Q \w_2, Q \w_3, Q \w_4\}$ which give the canonical rod diagram for $\txi_2$; the primitivity condition from \eqref{eq:primitivity} guarantees that existence of the matrix $Q$.

Consider now the gluing map between the two bundles. This map will operate between the subsets of $\txi_1$ and $\txi_2$ which are depicted by the overlap in Figure \ref{fig:plumbing2}. This region is an open neighborhood of a single corner, and thus is homeomorphic to $B^4 \times S^1$.
In both $\txi_1$ and $\txi_2$ the corner represents a single (polar) circle in the base $3$-manifold.
The overlap region can further be viewed as a trivialization $\B^2\times \D^2 \times S^1$ of the $\D^2$-bundles $\txi_1$, $\txi_2$
over a neighborhood of a polar circle. Here we use $B^2$ to denote a disk in the base, and $D^2$ to denote a disk in the fiber.
Just as in standard equivariant plumbing, Figure \ref{fig:plumbing2} shows that the $\D^2$ fibers in say $\txi_1$, which are represented by rays emanating from $\w_2$, switch roles in the overlap with the $\B^2$ sections in the base of $\txi_2$.
The gluing map is an automorphism on the overlap $\B^2\times \D^2 \times S^1$, and we have observed that the base and fiber disks $B^2$ and $D^2$ are exchanged in the gluing process. This leaves the circle $S^1$ unaccounted for. Since the automorphism must respect the action of $T^3$ on $\B^2 \times \D^2 \times S^1$, the image of this $S^1$ can be represented uniquely by an element of $\pi_1(T^3)\iso \Z^3$.
Note, however, that the image of $S^1$ in $\Z^3$ does not necessarily coincide with the polar circle, but rather an $S^1\sub T^3$ which acts upon it.
These circle actions are not unique as there are two Killing fields, the ones associated to $\B^2$ and $\D^2$, which vanish on the polar circle. The Lie group homomorphism from $T^3$ to $T^3$ arising from these circle actions should be an isomorphism.
This is the same as requiring that the image of the polar $S^1$, together with the circle actions on $\B^2$ and $\D^2$, forms an integral basis for $\Z^3$. The plumbing vector $\p_2 \in \Z^3$ may then be interpreted as representing the image of the polar circle, with the integral basis criteria being equivalent to the primitivity property \eqref{eq:primitivity}.

Writing a simple $T^n$-manifold as a toric plumbing of disk bundles $\plumb{\txi_1, \ldots, \txi_l \big| \p_2, \ldots, \p_l}$
facilitates the analysis of rod diagrams.
Indeed $\plumb{\txi_1, \ldots, \txi_l \big| \p_2, \ldots, \p_l}$ and $\plumb{\txi'_1, \ldots, \txi'_l \big| \p'_2, \ldots, \p'_l}$ can be distinguished easily, as they are equivariantly homeomorphic if and only if $\txi_j\iso \txi_j'$ and $\p_k=\p_k'$ for all $j$ and $k$.
To see this, use Proposition \ref{thrm:wp} to obtain rod structures $\{\w_1, \ldots, \w_{l+2}\}$ and $\{\w'_1, \ldots, \w'_{l+2}\}$ from the disk bundles and plumbing vectors. These rod structures are automatically in their unique Hermite normal form, and therefore the two simple $T^n$-manifolds are equivariantly homeomorphic if and only if the rod structures are identical.

\begin{remark}
Given a set of bundles $\{\txi_1, \ldots, \txi_l\}$, it may be difficult to determine all possible sets of vectors $\{\p_2, \ldots, \p_l\}$ for which the plumbing relations \eqref{eq:plumbrelations} are satisfied.
However, it is straightforward to check if a given set of vectors $\{\p_2, \ldots, \p_l\}$ satisfies the plumbing relations for the bundles $\{\txi_1, \ldots, \txi_l\}$. Namely, first confirm that each $\p_i$ is a primitive vector. Then simply follow the recursion equations \eqref{eq:recursion} to find all the $\w_i$. If each successive pair $\{\w_i, \w_{i+1}\}$ is admissible, that is, if their second determinant divisor is $1$, then $\{\w_1, \ldots, \w_{l+2}\}$ does indeed give a well defined rod diagram for a manifold. Lastly, check that $\{\w_1, \ldots, \w_{l+2}\}$ is in Hermite normal form. If so, then $\{\p_2, \ldots, \p_l\}$ are valid plumbing vectors for the manifold arising from $\{\w_1, \ldots, \w_{l+2}\}$.
\end{remark}

\begin{figure}[ht!]
	\centering 
\includegraphics[width=15.5cm]{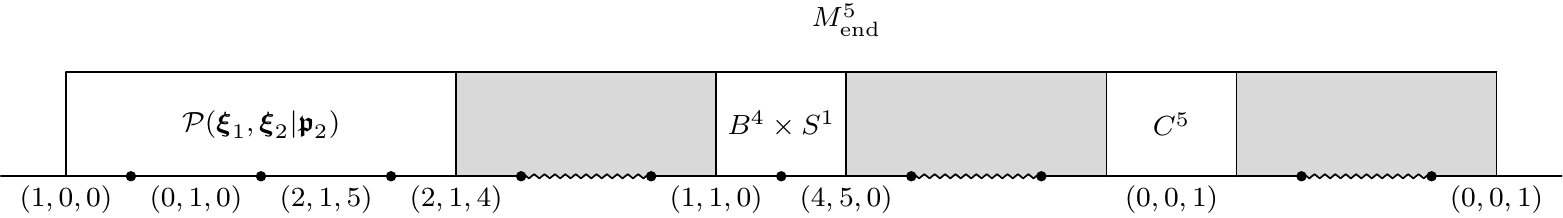}
\caption{This is an example of the decomposition of the domain of outer communication described in Theorem \ref{main:DOC}. The black hole horizons, represented by jagged intervals, are deformation retracts of the gray areas. In the leftmost piece of the decomposition, $\txi_1$ is formed by a disk bundle over $L(5,2)$ with Euler class determined by $1$, while $\txi_2$ is formed by a disk bundle over $L(2,1)$ with Euler class $0$; the plumbing vector is $\p_2 = (1,0,2)$. The remaining pieces include a neighborhood of a corner $B^4 \times S^1$, a region centered on the interior of an axis rod $C^5=[0,1]\times \D^2 \times T^2$, and the asymptotic end $M_{end}^5$ which is homeomorphic to $\R_+ \times S^3 \times S^1$.} 
\label{fig:DOCdecomp}
\end{figure}

The strategy to establish Theorem~\ref{main:DOC} is illustrated in Figure~\ref{fig:DOCdecomp}. More precisely, consider the orbit space of the domain of outer communication, and remove neighborhoods of the horizon rods (corresponding to the gray areas in the diagram). The axis is then broken into connected components, whose neighborhoods in the orbit space lift to one of the pieces in the total space of the decomposition \eqref{gfuiqihgkjso}. In particular, if the neighborhood contains no corners, one corner, or multiple corners then it is represented by $C^{n+2}_k$, $B_{m}^{4}\times T^{n-2}$, or $\plumb{\txi_{1,j},\ldots,\txi_{I_j,j} \big| \p_{2,j}, \ldots, \p_{I_j,j}}$ respectively. 
The remaining portion of the orbit space lifts to the asymptotic end. Clearly any rod diagram that arises from a DOC, with the current hypotheses, can be organized into such pieces. This completes the proof of Theorem~\ref{main:DOC}.

\section{Classification of Compact Spaces}\label{sec:class}

\renewcommand{\B}{B}
\renewcommand{\D}{D}

Theorem \ref{main:456classification} arises from the classification of compact simply connected $T^n$-manifolds of cohomogeneity two in dimensions 4, 5, and 6. In dimensions seven and higher, a complete classification is not known, and the technique used by Oh~\cite{Oh5dim,Oh6dim} in the lower dimensional cases does not appear to generalize to higher dimensions.
On the other hand, the fundamental groups of $(n+2)$-dimensional $T^n$-manifolds can be readily computed in all dimensions by the Seifert-Van Kampen theorem, as recorded in the next result. Note that a portion of part $(i)$ was established within the proof of Theorem 4 in~\cite{hollands2011}.

\begin{theorem}\label{thrm:simple}\

\begin{enumerate}[(i)]
\item Let $M^{n+2}$, $n\geq 1$ be a closed orientable manifold with an effective $T^n$-action. If $M^{n+2}$ is simply connected then it is either the 3-sphere, or a simple $T^n$-manifold where the integral span of its rod structures is $\Z^n$.

\item Let $M^{n+2}$ be a connected simple $T^n$-manifold, possibly with boundary. Suppose that the rod diagram that represents $M^{n+2}$ is given by rod structures $\{\v_1, \ldots, \v_m\}\sub \Z^n$. Then the fundamental group takes the form
\begin{equation}
\label{eq:pi1}
\pi_1(M^{n+2})\iso \Z^n/\spn_\Z\{\v_1, \ldots, \v_m\}
\iso \Z^{n-l} \oplus \Z_{s_1}\oplus \cdots \oplus \Z_{s_l},
\end{equation}
where $s_i | s_{i+1}$ and $s_i$ is the $i^\text{th}$ entry in the Smith normal form of the matrix composed of column vectors $\v_i$, and $l=\dim\spn_\R\{\v_1, \ldots, \v_m\}$.
\end{enumerate}
\end{theorem}

\begin{proof}
Consider part $(i)$. The fundamental group of a $T^n$-manifold of dimension $n+2$ can be calculated from the topology of the quotient space and the bundle structure, using the Seifert-Van Kampen theorem. This was carried out by Orlik and Raymond~\cite[Page 94]{Orlik1974} in the case when the quotient space is an orbifold without boundary, yielding the group presentation
\begin{equation}
\label{eq:pi1orlik}
\begin{split}
\pi_1(M^{n+2}) \iso
& \big\langle \tau_1, \ldots, \tau_n,
\alpha_1, \ldots, \alpha_a,
\gamma_1, \ldots, \gamma_g,
\delta_1, \ldots, \delta_g \big |  \\
&[\tau_i, \tau_j]; [\tau_i, \alpha_j];[\tau_i, \gamma_j];[\tau_i, \delta_j];
\quad \text{for all $i$ and $j$}\\
&[\gamma_1, \delta_1]\cdots [\gamma_g, \delta_g]
\cdot \alpha_1 \cdots \alpha_a
\cdot \tau_1^{c_1}\cdots \tau_n^{c_n};\\
&\alpha_l^{q_l}\cdot\tau_1^{p_{l1}}\cdots \tau_n^{p_{ln}};
\quad \text{for $l=1, \ldots, a$} \big\rangle.
\end{split}
\end{equation}
The generators $\tau$ arise from the torus fibers, the $\alpha$'s represent loops around each of the $a$ orbifold points, and the $\gamma$'s and $\delta$'s are generators associated with each of the $g$ handles. In the first line of relations we see that the $\tau$'s commute with themselves as they are the generators of a torus, and commute with the $\alpha$'s, $\gamma$'s, and $\delta$'s since the former are generators of the fiber and the latter are generators in base space $M^{n+2}/T^n$.
In analogy with the presentation of the fundamental group of a genus $g$ surface, the second line of relations represents the obstruction to contractibility of the circumscribing loop around all of the handles and orbifold points. That loop is homotopic to the loop around the fibers described by $\mathbf{c}=(c_1,\dots,c_n)\in \Z^n\iso \pi_1(T^n)$.
The last line of relations indicates how each orbifold point singularity is to be resolved, namely, going around the $i$-th orbifold point $q_i\neq 1$ times is equivalent to going around each of the torus fibers $p_{ij}$ times.

We wish to show in this case that $M^{n+2}\iso S^3$. To do that, let the list of generators in Equation \eqref{eq:pi1orlik} be denoted by $\G$ and the list of relations by $\mathcal{R}$, so that $\pi_1(M^{n+2})\iso \group{\G}{\mathcal{R}}$ is trivial.
Clearly then the group
$\mathcal H_1=\group{\G}{\mathcal{R}\cup \{[\alpha_i, \alpha_j], \gamma_k, \delta_k \}}$
is also trivial. This is an abelian group which can be presented as
\begin{equation}
\mathcal{H}_1 = \left(\Z^a\oplus \Z^n\right)/\spn_\Z\{(\mathbf{1},\mathbf{c}), (q_1 \mathbf{e}_1, \mathbf{p}_1), \ldots, (q_a \mathbf{e}_a ,\mathbf{p}_a) \},
\end{equation}
where $\mathbf{1}\in\Z^a$ is the vector consisting of all $1$'s and $\mathbf{p}_l =(p_{l1},\ldots,p_{ln})\in\mathbb{Z}^n$. 
The number of generators is $a+n$, and the number of relations is $a+1$, hence $\mathcal H_1$ can only be trivial if $n\leq 1$.
If $n=1$ then $M^{n+2}$ is a simply connected closed $3$-manifold, and thus is homeomorphic to $S^3$. 

We now consider the case where the quotient has boundary, that is $\partial \left(M^{n+2}/T^n\right)\ne \emptyset$. The fundamental group in this case was calculated by Hollands and Yazadjiev~\cite[Theorem 3]{hollands2011}, and takes the form
\begin{equation}
\label{eq:pi1hollands}
\begin{split}
\pi_1(M^{n+2}) \iso
& \big\langle \tau_1, \ldots, \tau_n, \alpha_1, \ldots, \alpha_a, \beta_1, \ldots, \beta_b, \gamma_1, \ldots, \gamma_g, \delta_1, \ldots, \delta_g \big |  \\
&[\tau_i, \tau_j]; [\tau_i, \alpha_j];[\tau_i, \beta_j];[\tau_i, \gamma_j];[\tau_i, \delta_j];
\quad \text{for all $i$ and $j$}\\
&[\gamma_1, \delta_1]\cdots [\gamma_g, \delta_g] \cdot \alpha_1 \cdots \alpha_a \cdot \beta_1\cdots \beta_b;\\
&\alpha_l^{q_l}\cdot\tau_1^{p_{l1}}\cdots \tau_n^{p_{ln}};
\quad \text{for $l=1, \ldots, a;$}\\
&\tau_1^{v_{k1}}\cdots \tau_n^{v_{kn}};
\quad\text{for $k=1, \ldots, m$} \big\rangle.
\end{split}
\end{equation}
The extra generators $\beta$ represent the $b$ boundary components of the orbit space which are homeomorphic to circles; on these components the torus action does not degenerate. Additional relations are included for these generators showing that they commute with the generators of the torus fibers. Moreover, the last line of relations is given by rod structures $\{\v_1, \ldots, \v_m\}$ for $M^{n+2}$ where each $\v_k=(v_{k1}, \ldots, v_{kn})$ represents a generator of the isotropy subgroup along the corresponding rod.
As before denote the generators of \eqref{eq:pi1hollands} by $\G$ and the list of relations by $\mathcal{R}$.
We can immediately determine that $g=0$ by examining $\group{\G}{\mathcal{R}\cup\{\tau_i, \alpha_j, \beta_\ell\}}$, which is in fact the fundamental group of a genus $g$ surface.
Next consider the subgroup $\group{\G}{\mathcal{R} \cup \{\tau_i, \alpha_j\}} = \group{\beta_1, \ldots, \beta_b}{\beta_1\cdots \beta_b}$, and observe that it is trivial only when all $\beta_i =1$, or rather $b=1$. Now consider the abelian group $\mathcal{H}_2 = \group{\G}{\mathcal{R}\cup\{\tau_i, [\alpha_i, \alpha_j]\}}$, which may be presented as
\begin{equation}
\mathcal H_2 = \Z^a/\spn_\Z \{ \mathbf{1}, q_1 \mathbf{e_1}, \dots, q_a \mathbf{e_a}\}.
\end{equation}
This group cannot be trivial unless $q_1=\dots=q_a=1$, however this contradicts the nature of $q_i$, and thus $a=0$. We then find that
\begin{equation}
\group{\G}{\mathcal{R}} = \Z^n/\spn_\Z\{\v_1, \ldots , \v_m\},
\end{equation}
and note that this is trivial only if the integral span of the rod structures is $\Z^n$.

Lastly, we will establish part $(ii)$. 
Notice that Equation~\eqref{eq:pi1hollands} reduces to the first equality in Equation~\eqref{eq:pi1} when $M^{n+2}$ is a simple $T^n$-space, since in this situation $M^{n+2}/T^n$ has no holes, handles, or orbifold points.
Furthermore, recall that the Smith normal form of the matrix $( \v_1, \v_2, \ldots, \v_m )$ is obtained by both left and right actions using unimodular matrices. This does not alter the integral span of the columns. Thus, as in the classification of finitely generated abelian groups, by a change of basis given by these unimodular matrices, we obtain the second equality in~\eqref{eq:pi1}.
\end{proof}

Theorem \ref{thrm:simple} may be used to as a tool to analyze the topology of the domain of outer communication for stationary vacuum $n$-axisymmetric spacetimes. A conjecture providing a topological classification of the DOCs in the asymptotically Kaluza-Klein setting, and under a spin assumption, has been put forth by Hollands-Ishibashi in \cite[Conjecture 1]{hollands2012}. We now recall the original statement.

\begin{conjecture*}[Hollands-Ishibashi]
Assume that $\mathcal{M}^{n+3}$, $n\geq 2$ is the domain of outer communication of a well-behaved asymptotically flat or asymptotically Kaluza-Klein spacetime which is spin, has Ricci tensor satisfying the null-convergence condition, and admits an effective $U(1)^{n}$ action. Then any Cauchy surface $M^{n+2}$ can be decomposed as
\begin{equation}\label{aognwoihj}
M^{n+2} \iso \of{ \#^{n}_{i=2} m_i \cdot \of{S^i \times S^{n+2-i} } \# (\text{Asymptotic Region} )}\setminus \text{Black Holes},
\end{equation}	
where the asymptotic region depends on the precise boundary conditions; e.g. in the standard Kaluza-Klein setup $\R^3\times T^{n-1}$.
\end{conjecture*}

This conjecture
implies that the fundamental group for the Cauchy surface always agrees with the fundamental group of the asymptotic region. Indeed, recall that taking a connected sum with simply connected space $S^k \times S^{n+2-k}$ does not affect the fundamental group, and neither does removing the black hole regions as can be seen from topological censorship, or alternatively by using Theorem \ref{thrm:simple}. The next proposition provides an explicit static vacuum counterexample to the above conjecture.

\begin{prop}
\label{prop:counter example}
There exists a well-behaved asymptotically Kaluza-Klein static bi-axisymmetric vacuum spacetime $\mathcal{M}^5 =\mathbb{R}\times M^4$, which is devoid of conical singularities and has two spherical horizons. The domain of outer communication is spin and simply connected, while its asymptotic region is not simply connected. In particular, the Cauchy surface $M^4$ violates Conjecture 1 of \cite{hollands2012}.
\end{prop}

\begin{proof}
Consider the rod diagram consisting of rod structures $\{(1,0), (0,0), (0,1), (0,0), (1,0)\}$.
According to Theorem \ref{theorem1}, there exists a well-behaved asymptotically Kaluza-Klein static bi-axisymmetric vacuum spacetime $\mathcal{M}^5 =\mathbb{R}\times M^4$, whose orbit space $M^4/T^2$ is a half-plane admitting this rod diagram. The two $(0,0)$ rods represent $S^3$ horizons, and the two semi-infinite rods $(1,0)$ give rise to the asymptotically Kaluza-Klein end $M^4_{end}\cong \mathbb{R}^3 \times S^1$. Moreover, in \cite[Section 6]{khuriweinsteinyamada2018} it is shown that there are no conical singularities on the two semi-infinite rods. The spacetime metric may be expressed in Weyl-Papapetrou form as in \eqref{weyl}. Furthermore, since the Killing field $\partial_{\phi^2}$ that degenerates on the middle axis rod $(0,1)$ does not affect the cone angle at the two semi-infinite rods, or the asymptotics in $M^{4}_{end}$ other than the size of the $S^1$ factor, we may scale the $\phi^2$ coordinate appropriately to relieve any angle defect on this rod. The spacetime is then regular. 

We will now analyze the topology of the domain of outer communication. First observe that Theorem \ref{thrm:simple} implies that $M^4$ is simply connected, while clearly $\pi_1(M^4_{end})=\mathbb{Z}$.
Next, fill in each $S^3$ horizon with a 4-ball $B^4$. This may be accomplished in the rod diagram by connecting the rods flanking the horizons with a single corner. As for the asymptotic end, a cross-section has the topology $S^1 \times S^2$, and thus may be filled in with an $S^1 \times B^3$. The asymptotic end is flanked by the rods $(1,0)$ and $(1,0)$, and thus the filling may be achieved in the rod diagram by extending one of these semi-infinite axis rods until it reaches the other, so that a single axis rod with the same rod structure is formed out of the two semi-infinite rods. Note that these fill-ins respect the $T^2$-structure by construction. After filling in the horizons and capping off the asymptotic end, we are left with a closed simple $T^2$-manifold having a rod diagram consisting of only two axis rods of rod structures $(1,0)$ and $(0,1)$, which meet at two admissible corners. This is the rod diagram for $S^4$. Therefore, the DOC $M^4$ is homeomorphic to $S^4 \setminus (\B^4 \sqcup \B^4 \sqcup S^1\times \B^3)$ which is homotopic to $\R^4\setminus (\{\text{pt.}\} \sqcup S^1)$, which is a spin manifold.

Now assume by way of contradiction that Conjecture 1 of \cite{hollands2012} is true. Although the black hole region is unknown, it cannot intersect the asymptotic region, by definition. We can therefore rearrange terms in \eqref{aognwoihj} to find $M^4 \iso \of{ \of{\# m_2 \cdot S^2 \times S^2 }\setminus \text{Black Holes}} \# \of{\text{Asymptotic Region}}$. Recall that in three or more dimensions, the fundamental group of a connected sum is the free product of the fundamental groups of its components. Moreover, as stated in the conjecture, the asymptotic region for the standard Kaluza-Klein setup is $\R^3\times S^1$. Therefore, there is an injective homomorphism $\Z\iso \pi_1(\R^3\times S^1)\into \pi_1 (M^4)$. This leads to a contradiction, since we have already seen that $M^4 \iso \R^4\setminus ( \{\text{pt.}\} \sqcup S^1)$, which is simply connected.
\end{proof}

Even though Conjecture 1 of \cite{hollands2012} is not true as stated, the spirit of the conjecture
which suggests that in the spin case Cauchy surfaces are primarily comprised of connected sums of products of spheres, may nevertheless remain valid. In fact Theorem \ref{main:456classification}, which will be proven at the end of this section, confirms this sentiment in low dimensions. We are thus motivated to formulate a refined version, Conjecture \ref{main:refined}, and will give a proof of this conjecture for spacetime dimensions $5$, $6$, and $7$. The primary difference between the revised and original versions is that instead of removing the black hole regions and including a connected sum to the asymptotic end, we consider closed extensions $\bar{M}^{n+2}\supset M^{n+2}\setminus M^{n+2}_{end}$. These extensions, which may be viewed as compactified domains of outer communication, fill in the asymptotic region as well as every horizon to form a closed manifold. Theorems \ref{thrm:fillin} and \ref{thrm:simple} show that it is always possible to perform such fill-ins and obtain a closed, simply connected $T^n$-manifold, albeit the compactified DOC $\bar{M}^{n+2}$ may not be spin.

\begin{prop}
\label{thrm:conjecture proof}
Conjecture \ref{main:refined} is valid when $n=2$, $3$, or $4$, if the compactified domain of outer communication is spin.
\end{prop}

\begin{proof}
Let $M^{n+2}$ be a Cauchy surface for the domain of outer communication of the spacetime $\mathcal{M}^{n+3}$ satisfying the desired hypotheses. Since all Cauchy surfaces are homeomorphic, we can without loss of generality assume that $M^{n+2}$ admits a $U(1)^n$ symmetry. This, together with the topological censorship theorem, shows that $M^{n+2}$ is a simple $T^n$-manifold~\cite[Theorem 9]{hollands2012}. To construct the compactified DOC $\bar{M}^{n+2}\supset M^{n+2}\setminus M^{n+2}_{end}$, we cap off the asymptotic region and fill in all of the horizons in such a way that the total space is simply connected, by adding additional rods. Theorem \ref{thrm:fillin} describes how to construct the fill-ins from the rod diagram, while Equation \eqref{eq:pi1} explains how to make the total space simply connected. If $n=2$, $3$, or $4$, and if $\bar{M}^{n+2}$ is spin, then by Theorem \ref{main:456classification} it is homeomorphic to a connect sum of products of spheres.
\end{proof}

It is likely the case that a spin DOC yields a spin compactified DOC in the proof of this proposition, in which case Conjecture \ref{main:refined} would be fully verified for $n=2$, $3$, or $4$. Furthermore, Proposition \ref{thrm:conjecture proof} can be generalized to include the non-spin case where $\bar{M}^{n+2}$ will instead be homeomorphic to a manifold in the third row of the table from Theorem \ref{main:456classification}. In addition, it should be noted that the refined conjecture can be extended to the setting where geometric regularity of the spacetime metric is not required. This is relevant to applications of Theorem \ref{theorem1}, since generic spacetimes produced by this result may include conical singularities on the axes.

\begin{remark}
A slightly modified version of Proposition \ref{thrm:conjecture proof} holds true when the spacetime $\mathcal{M}^{n+3}$ has conical singularities on its axis rods. To see this, observe that the only place where geometric regularity of the metric becomes relevant, is when the topological censorship theorem is utilized. Thus, the regularity assumption as well as the null energy condition may be removed from the hypotheses of Conjecture \ref{main:refined}, if the topological censorship principle is added in their place.
This principle, together with the $U(1)^n$ symmetry, guarantees that the Cauchy surface $M^{n+2}$ is a simple $T^n$-manifold. The remaining portion of the proof then proceeds without change. In fact, the conjecture is at its core a purely topological statement.
\end{remark}

\begin{mainConjecture}
\label{conj:top}
Let $n\geq 1$. Any closed, spin, simply connected $(n+2)$-manifold with an effective $T^n$-action is homeomorphic to either $S^3$, $S^4$, $S^5$, or $\#^{n}_{i=2} m_i \cdot S^i\times S^{n+2-i}$.
\end{mainConjecture}

It does not appear that this conjecture has previously been recorded in the literature. However,
it should be noted that McGavran claimed in~\cite[Theorem 3.6]{McGavran79} (see also \cite{McGavran77}) to have proven a similar statement. Oh \cite{Oh5dim} pointed out flaws in McGavran's argument, and in fact provided counterexamples to his claims. Oh's  work on this topic \cite{Oh5dim,Oh6dim}, along with Orlik and Raymond's classification \cite{orlik1970actions} in the 4-dimensional case, remains the best evidence towards Conjecture~\ref{conj:top}.

\begin{proof}[Proof of Theorem \ref{main:456classification}]
We may follow the same line of argument as in the proof of Proposition \ref{thrm:conjecture proof}. In particular, by applying Theorems \ref{thrm:fillin} and \ref{thrm:simple} to cap-off the asymptotic end and fill-in the horizons, we arrive at a compactified domain of outer communication $\bar{M}^{n+2}$ which is closed, simply connected, and admits an effective $T^n$-action. Moreover, this process of capping-off and filling-in may be accomplished in an algorithmic manner, as explained in the proof of Theorem \ref{thrm:fillin}. We may then apply the classification results for such manifolds given in \cite{Oh5dim,Oh6dim,orlik1970actions} for $n=2,3,4$, to obtain the chart presented in Theorem \ref{main:456classification}.
\end{proof}

\bibliographystyle{amsplain}
\bibliography{mybib}
\end{document}